\documentclass[a4paper,reqno,10pt]{amsart}
\makeatletter
\newcommand*{\rom}[1]{\expandafter\@slowromancap\romannumeral #1@}
\makeatother
\usepackage{epsf,graphicx,epsfig}
\usepackage{amscd}
\usepackage{amsmath,latexsym,amssymb,amsthm}
\usepackage{fdsymbol}
\usepackage[nospace,noadjust]{cite}
\usepackage{textcomp}
\usepackage{nccmath}
\usepackage{setspace,cite}
\usepackage{mathrsfs,amsmath}
\usepackage{chemarrow}
\usepackage{lscape,fancyhdr,fancybox}
\usepackage{stmaryrd}
\usepackage[all,cmtip]{xy}
\usepackage{tikz-cd}
\usepackage{cancel}
\usepackage[usestackEOL]{stackengine}
\usetikzlibrary{shapes,arrows,decorations.markings}
\usepackage{graphicx}

\usepackage{color}
\usepackage{bbm, dsfont}
\usepackage{xcolor}
\usepackage{url}
\usepackage{enumerate}
\usepackage[mathscr]{euscript}
  \theoremstyle{plain}
\swapnumbers
    \newtheorem{thm}{Theorem}[section]
    \newtheorem{proposition}[thm]{Proposition}
     \newtheorem{theorem}[thm]{Theorem}

    \newtheorem{subsec}[thm]{}
\theoremstyle{definition}
    \newtheorem{definition}[thm]{Definition}
        \newtheorem{remark}[thm]{Remark}
\theoremstyle{remark}

\title{}
\author{}
\date{}
\usepackage{amssymb}

\usepackage{hyperref}
\hypersetup{
colorlinks,
citecolor=blue,
filecolor=black,
linkcolor=blue,
urlcolor=black
}

\newcommand{\g}{\vartriangleright}
\newcommand{\gdesh}{\vartriangleright '}
\newcommand{\gfill}{\blacktriangleright}
\newcommand{\lt}{\vartriangleleft}
\newcommand{\ldesh}{\vartriangleleft '}
\newcommand{\lfill}{\blacktriangleleft}

\usepackage[left=25mm,right=25mm,top=25mm,bottom=25mm,paper=a4paper]{geometry}

\begin{document}

%\title[]{New graded Lie algebras associated to a nonsymmetric operad with multiplication \\ Fr\"{o}licher-Nijenhuis bracket and derived bracket associated to a nonsymmetric operad with multiplication}

%\title[Fr\"{o}licher-Nijenhuis bracket and derived bracket]{Fr\"{o}licher-Nijenhuis bracket and derived bracket associated to a nonsymmetric operad with multiplication}

\title[]{Fr\"{o}licher-Nijenhuis bracket and derived bracket associated to a nonsymmetric operad with multiplication}

\author{Anusuiya Baishya}
\address{Department of Mathematics,
Indian Institute of Technology, Kharagpur 721302, West Bengal, India.}
\email{anusuiyabaishya530@gmail.com}

\author{Apurba Das}
\address{Department of Mathematics,
Indian Institute of Technology, Kharagpur 721302, West Bengal, India.}
\email{apurbadas348@gmail.com, apurbadas348@maths.iitkgp.ac.in}

\begin{abstract}
This paper aims to construct two graded Lie algebras associated with a nonsymmetric operad with multiplication. Maurer-Cartan elements of these graded Lie algebras correspond respectively to Nijenhuis elements and Rota-Baxter elements for the given multiplication. Explicit forms of these brackets are given to study Nijenhuis operators and Rota-Baxter operators on some Loday-type algebras.
\end{abstract}

%Among others, we explicitly write the graded Lie algebras that characterize Nijenhuis operators, Rota-Baxter operators and averaging operators on various Loday-type algebras and Hom-algebras.
%The purpose of this paper is to construct two new graded Lie algebras associated with a nonsymmetric operad with multiplication. The first graded Lie algebra is made to study Nijenhuis elements while the second one is useful to study Rota-Baxter elements on the underlying operad with multiplication.

%Finally, we introduce and study NS-dendriform algebras as a generalization of quadri algebras. \textcolor{red}{Tensor product of dendriform and NS}

\maketitle

\medskip

\medskip

    

\noindent

\thispagestyle{empty}

\tableofcontents

\vspace{0.2cm}

\section{Introduction}
    In a classical work, Gerstenhaber \cite{gers-ring} showed that the graded space of all multilinear maps on a vector space $A$ carries a graded Lie bracket (known as the {\em Gerstenhaber bracket}) that generalizes the commutator Lie bracket on the space of linear maps on $A$. A bijective correspondence exists between associative algebra structures on $A$ and Maurer-Cartan elements for the Gerstenhaber bracket. If $A$ is equipped with an associative algebra structure (i.e., endowed with a Maurer-Cartan element for the Gerstenhaber bracket) then the graded space of multilinear maps (which are the Hochschild cochains) also carries a {\em cup-product}. It has been observed in \cite{gers-ring} that the Gerstenhaber bracket and the cup-product pass onto the Hochschild cohomology level of the associative algebra. Further, these operations make the graded space of Hochschild cohomology a rich structure known as the {\em Gerstenhaber algebra}. Subsequently, Gerstenhaber and Voronov \cite{gers-voro} discovered this rich structure on Hochschild cohomology of an associative algebra in terms of a nonsymmetric operad endowed with a multiplication.

    \medskip

    A nonsymmetric operad $\mathcal{P}$ is a triple $\mathcal{P} = (\mathcal{P}, \circ, \mathbbm{1})$ consisting of a collection $\mathcal{P} = \{ \mathcal{P}_n \}_{n \geq 1}$ of vector spaces equipped with bilinear maps $\circ_i: \mathcal{P}_m \times \mathcal{P}_n \rightarrow \mathcal{P}_{m+n-1}$ (for $m, n \geq 1$ and $1 \leq i \leq m$), called partial compositions that satisfy certain associative-like identities and there exists an element $\mathbbm{1} \in \mathcal{P}_1$ that plays the role of identity for the partial compositions. A first example is given by the endomorphism operad $\mathcal{P} = \mathrm{End}_A$ associated with a vector space $A$ \cite{loday-val}. A nonsymmetric operad $\mathcal{P}$ naturally induces a graded Lie bracket (called the {\em Gerstenhaber-Voronov bracket}) $[~,~]_\mathsf{GV}$ on the graded space $\mathcal{P}_{\bullet + 1} = \bigoplus_{n \geq 0} \mathcal{P}_{n+1}$ that generalizes the Gerstenhaber bracket. Given a nonsymmetric operad $\mathcal{P}$, an element $\pi \in \mathcal{P}_2$ is called a {\em multiplication} if it satisfies $\pi \circ_1 \pi = \pi \circ_2 \pi$. When $\mathcal{P} = \mathrm{End}_A$ is the endomorphism operad associated with a vector space $A$, any multiplication is nothing but an associative algebra structure on $A$. A multiplication $\pi$ on a nonsymmetric operad $\mathcal{P}$ induces a cup-product $\cup_\pi$ and a differential $\delta_\pi$ on the graded space $\mathcal{P}_\bullet = \bigoplus_{n \geq 1} \mathcal{P}_{n}$. Further, $\delta_\pi$ is a graded derivation for the cup-product $\cup_\pi$. Just like an associative algebra, there are several other algebraic structures  (e.g. dendriform algebra \cite{loday-di}, tridendriform algebra \cite{loday-ronco}, diassociative algebra \cite{loday-di}, quadri-algebra \cite{aguiar-loday}, ennea algebra \cite{leroux}, NS-algebra \cite{uchino}, Hom-associative algebra \cite{das}, associative conformal algebra \cite{bakalov,hou} etc.) that can be described by nonsymmetric operads endowed with multiplication. All these algebras are often referred to as `Loday-type algebras' \cite{yau, yu, das}. It is important to note that Lie algebras, Leibniz algebras or any algebras whose defining identity/identities have shufflings cannot be described by multiplications in nonsymmetric operads.

\medskip

Our primary aim in this paper is to construct two new graded Lie brackets associated with a nonsymmetric operad with multiplication. The first one is the {\em Fr\"{o}licher-Nijenhuis bracket} that generalises the classical Fr\"{o}licher-Nijenhuis bracket \cite{fro-nij-1,fro-nij-2} defined on the space of all vector-valued differential forms. The second one is the {\em derived bracket} which generalises the well-known derived bracket \cite{das-rota} defined on the Hochschild cochains of an associative algebra characterizing Rota-Baxter operators as its Maurer-Cartan elements. Finally, we will explicitly write these brackets for some Loday-type algebras.

%associative algebras, dendriform algebras and tridendriform algebras and find the explicit form of the Fr\"{o}licher-Nijenhuis bracket and the derived bracket for such algebras.% In the end, we also introduce and study a new algebraic structure, called {\em NS-dendriform algebra} that generalizes quadri-algebra. Our steps in constructing the brackets and explicit results are given in the following.

\medskip

    Let $\mathcal{P}$ be a nonsymmetric operad and $\pi$ be a multiplication on $\mathcal{P}$. Since the cup-product $\cup_\pi$ is graded associative, its graded commutator yields a graded Lie bracket (called the {\em cup bracket}) $[~,~]_\pi$ on $\mathcal{P}_\bullet$. We find some new descriptions of the bracket $[~,~]_\pi$ which are useful in proving our main results. We also find a differential $D$ on the graded space $\mathcal{P}_\bullet$ that makes $(\mathcal{P}_\bullet, [~,~]_\pi, D)$ into a differential graded Lie algebra. An element $\varphi \in \mathcal{P}_1$ is a Maurer-Cartan element of this differential graded Lie algebra if and only if
    \begin{align*}
         \varphi \circ_1 \pi= (\pi \circ_2 \varphi) \circ_1 \varphi.
    \end{align*}
    Such an element $\varphi$ is said to preserve the multiplication $\pi$. When $\mathcal{P}$ is the endomorphism operad $\mathrm{End}_A$ and the multiplication $\pi$ corresponds to an associative algebra structure on $A$, then an element $\varphi \in (\mathrm{End}_A)_1 = \mathrm{Hom}(A, A)$ preserves the multiplication $\pi$ if and only if it is an associative algebra {homomorphism}.

    \medskip
    
    Next, given a nonsymmetric operad $\mathcal{P}$ with a multiplication $\pi$, we first define {\em Nijenhuis elements} for the multiplication $\pi$ as a generalization of Nijenhuis operators on associative algebras \cite{grab}. We show that a Nijenhuis element induces a hierarchy of pairwise compatible Nijenhuis elements generalizing a well-known result about classical Nijenhuis operators. Next, we find an action of the Gerstenhaber-Voronov algebra $(\mathcal{P}_{\bullet +1} , [~,~]_\mathsf{GV})$ on the graded Lie algebra $(\mathcal{P}_\bullet, [~,~]_\pi)$ by graded derivations. This result yields the semidirect product graded Lie algebra structure on the direct sum $\mathcal{P}_{\bullet + 1} \oplus \mathcal{P}_\bullet$. For each $m, n \geq 1$, we then define a bracket $[~,~]_\mathsf{FN} :  \mathcal{P}_m \times  \mathcal{P}_n \rightarrow  \mathcal{P}_{m+n}$ by
        \begin{align*}
            [f,g]_\mathsf{FN}:=[f,g]_\pi + (-1)^m ~ \!  \iota_{(\delta_\pi f)} g - (-1)^{(m+1)n} ~\!  \iota_{(\delta_\pi g)} f,
        \end{align*}
    for $f \in  \mathcal{P}_m$ and $g \in  \mathcal{P}_n$. Here $~\iota~$ is the contraction operator. The above bracket is graded skew-symmetric and it satisfies the graded Jacobi identity (that can be shown by using the graded Jacobi identity of the semidirect product bracket defined on  $\mathcal{P}_{\bullet + 1} \oplus \mathcal{P}_\bullet$). The bracket $[~,~]_\mathsf{FN}$ is called the {\em Fr\"{o}licher-Nijenhuis bracket} and the graded Lie algebra $(\mathcal{P}_\bullet, [~,~]_\mathsf{FN})$ is called the {\em Fr\"{o}licher-Nijenhuis algebra} associated to the nonsymmetric operad $\mathcal{P}$ with the multiplication $\pi$. Then an element $N \in \mathcal{P}_1$ is a Maurer-Cartan element in the Fr\"{o}licher-Nijenhuis algebra $(\mathcal{P}_\bullet, [~,~]_\mathsf{FN})$ if and only if $N$ is a Nijenhuis element for the multiplication $\pi$. Using this characterization, we also define the cohomology of a distinguished Nijenhuis element $N$.
    %This shows that Nijenhuis elements for the multiplication $\pi$ are exactly the Maurer-Cartan elements in the Fr\"{o}licher-Nijenhuis algebra  $(\mathcal{P}_\bullet, [~,~]_\mathsf{FN})$.

\medskip

In the next part, given a nonsymmetric operad $\mathcal{P}$ with a multiplication $\pi$, we first define {\em Rota-Baxter elements of weight $\lambda \in {\bf k}$} for the multiplication $\pi$ as a generalization of Rota-Baxter operators \cite{das-rota,guo-book}. For any $m, n \geq 1$, we define another bracket $[~,~]_\mathsf{D} : \mathcal{P}_m \times \mathcal{P}_n \rightarrow \mathcal{P}_{m+n}$ by
        \begin{align*}
            [f,g]_\mathsf{D} := [f,g]_\pi + \iota_{\theta f} g - (-1)^{mn} ~\!  \iota_{\theta g} f, 
        \end{align*}
    for $f \in \mathcal{P}_m$ and $g \in \mathcal{P}_n$. Here the map $\theta$ is given in (\ref{theta}). This bracket is graded skew-symmetric and also satisfies the graded Jacobi identity. We call $[~,~]_\mathsf{D}$ the {\em derived bracket} and the graded Lie algebra $(\mathcal{P}_\bullet, [~,~]_\mathsf{D})$ is called the {\em derived algebra} associated to the nonsymmetric operad $\mathcal{P}$ with the multiplication $\pi$. Next, for any given scalar $\lambda \in {\bf k}$, we define a map $d_\lambda : \mathcal{P}_n \rightarrow \mathcal{P}_{n+1}$ (for any $n \geq 1$) by
        \begin{align*}
             d_\lambda f = -\lambda ~ \! \iota_\pi f, \text { for } f \in \mathcal{P}_n.
        \end{align*}
    The map $d_\lambda$ turns out to be a differential (i.e., $d_\lambda^2 = 0$) and a graded derivation for the derived bracket $[~,~]_\mathsf{D}$. In other words, $(\mathcal{P}_\bullet, [~,~]_\mathsf{D}, d_\lambda)$ becomes a differential graded Lie algebra (dgLa). We observed that an element $R \in \mathcal{P}_1$ is a Maurer-Cartan element in the graded Lie algebra $(\mathcal{P}_\bullet, [~,~]_\mathsf{D})$ (resp. dgLa $(\mathcal{P}_\bullet, [~,~]_\mathsf{D}, d_\lambda)$ ) if and only if $R$ is a Rota-Baxter element of weight $0$ (resp. of weight $\lambda \in {\bf k}$) for the multiplication $\pi$. Subsequently, we define the cohomology of a distinguished Rota-Baxter element $R$ of any weight $\lambda \in {\bf k}$. We show that the cohomology of $R$ can be realized as the cohomology of the induced multiplication with coefficients in a suitable representation. Subsequently, we also consider {\em averaging elements} for the multiplication $\pi$ as a generalization of averaging operators on associative algebras \cite{das-avg,pei-guo}. Using the help of planar binary trees, we construct another graded Lie algebra whose Maurer-Cartan elements are precisely averaging elements for the multiplication $\pi$.

\medskip

As mentioned earlier, an associative algebra structure on a vector space $A$ is equivalent to having a multiplication on the endomorphism operad $\mathrm{End}_A$. The Fr\"{o}licher-Nijenhuis bracket (resp. the derived bracket) for an associative algebra $A$ is defined to be the Fr\"{o}licher-Nijenhuis bracket (resp. the derived bracket) associated with the endomorphism operad $\mathrm{End}_A$ with the multiplication that corresponds to the given associative structure on $A$. For an associative algebra, we provide the explicit descriptions of these brackets. It is also well-known that a dendriform algebra structure is equivalent to having a multiplication in a suitable nonsymmetric operad \cite{das}. Hence, we may define the Fr\"{o}licher-Nijenhuis bracket and the derived bracket for a given dendriform algebra. We also explicitly write these brackets for a Hom-associative algebra. As demonstrated in \cite{hou}, an associative conformal algebra can be represented using a nonsymmetric operad with multiplication, which allows for the definition of these brackets within that framework.

\medskip

        % \[
         %    \xymatrix{
          %   [~,~]_\mathsf{FN} : \mathrm{Hom} (A^{\otimes m}, A) \times \mathrm{Hom} (A^{\otimes n}, A) \ar[r] \ar[d] & \mathrm{Hom} (A^{\otimes m+n}, A) \ar[d] \\
           %  [~,~]_\mathsf{FN} : \mathrm{Hom} (A^{\otimes m}, A) \times \mathrm{Hom} (A^{\otimes n}, A) \ar[r] & \mathrm{Hom} (A^{\otimes m+n}, A) 
       % }
       % \}

   % \begin{remark}
   
   In this paper, we study Nijenhuis elements, Rota-Baxter elements and averaging elements as Maurer-Cartan elements in suitable graded Lie algebras. In particular, we explicitly write the Fr\"{o}licher-Nijenhuis bracket and the derived bracket for various associative-like algebras to study Nijenhuis operators and Rota-Baxter operators on them. However, these operators can be defined on other algebraic structures (e.g. Lie algebras, Leibniz algebras) which are not described by nonsymmetric operads equipped with multiplication. In future work, we will study these operators on algebras over any binary quadratic operad and construct the corresponding Fr\"{o}licher-Nijenhuis bracket and derived bracket.
   % \end{remark}

   \medskip

All vector spaces, tensor products, linear and multilinear maps are over a field ${\bf k}$ of characteristics $0$ unless specified otherwise.

\medskip

\section{Nonsymmetric operad with multiplication}\label{sec2}

    This section aims to recall some basic definitions and results about nonsymmetric operads endowed with multiplications. Our main references are \cite{gers-ring,gers-voro,loday-val}.

    \begin{definition} \label{nonsymm}
         A {\bf nonsymmetric operad} (also called {\bf non-$\Sigma$ operad}) is a triple $\mathcal{P} = (\mathcal{P}, \circ , \mathbbm{1})$ consisting of a collection $\mathcal{P}=\{\mathcal{P}_n\}_{n \geq 1}$ of vector spaces equipped with bilinear maps 
            \begin{align} \label{compo}
                \circ_i : \mathcal{P}_m \times \mathcal{P}_n \rightarrow \mathcal{P}_{m+n-1} \quad \text{(for  $1 \leq i \leq m$)}
            \end{align}
        that satisfy the following conditions: for $ f\in \mathcal{P}_m$, $g \in \mathcal{P}_n$ and $h \in \mathcal{P}_k$,
            \begin{align*}
                f \circ_i (g \circ_j h) & = (f \circ_i g) \circ_{i+j-1} h,  \text{ for } 1\leq i \leq m, 1 \leq j \leq n, \\
                (f \circ_i g) \circ_{j+n-1} h & = (f \circ_j h) \circ_{i} g, \text{ for } 1 \leq i < j \leq m
            \end{align*}
        and there is an element $\mathbbm{1} \in \mathcal{P}_1$ such that $f \circ_i \mathbbm{1}= \mathbbm{1} \circ_1 f = f $, for all $f \in \mathcal{P}_m$ and $1 \leq i \leq m$.
    \end{definition}

    The maps given in (\ref{compo}) are called `partial compositions' and Definition \ref{nonsymm} is called the `partial definition' of a nonsymmetric operad. This definition first appeared in a paper by Gerstenhaber \cite{gers-ring} under the name of `pre-Lie system'. We refer \cite{loday-val} for some other equivalent definitions of a nonsymmetric operad.
    
\begin{remark}
Nonsymmetric operads encode those types of algebras for which the generating operations have no symmetry and the defining relations are multilinear in which the variables stay in the same order. A nonsymmetric operad  $(\mathcal{P}, \circ, \mathbbm{1})$ is said to be a symmetric operad (or simply an operad) if for each $n \geq 1$, the space $\mathcal{P}_n$ is endowed with a right $\mathbf{k}[\mathbb{S}_n]$-module structure and all such actions are compatible with the partial compositions. A generic type of algebras is encoded by symmetric operads.
\end{remark}

    A toy example of a nonsymmetric operad is given by the endomorphism operad associated with a vector space $A$. Recall that the endomorphism operad $\mathrm{End}_A$ is given by $(\mathrm{End}_A)_n := \mathrm{Hom}(A^{\otimes n}, A)$ for $n \geq 1$ and the partial compositions are given by 
        \begin{align} \label{ass-circ}
            ( f \circ_i g)(a_1, \ldots , a_{m+n-1})= f(a_1, \ldots ,a_{i-1}, g (a_i, \ldots , a_{i+n-1}), \ldots , a_{m+n-1}),
        \end{align}
    for $f \in (\mathrm{End}_A)_m, ~ \! g \in  (\mathrm{End}_A)_n$ and $a_1,\ldots, a_{m+n-1} \in A$. The identity map $\mathrm{Id} \in (\mathrm{End}_A)_1$ is the identity element.

    Let $\mathcal{P}$ be a nonsymmetric operad. For each $m,n \geq 1$, define a bilinear operation (called the {\em contraction operator}) $\iota : \mathcal{P}_m \times \mathcal{P}_n \rightarrow \mathcal{P}_{m+n-1}$, $(f, g) \mapsto \iota_g f$ by
        \begin{align*}
            \iota_g f := \sum_{i=1}^m (-1)^{(i-1)(n-1)}~ \! f \circ_i g, \text{ for } f \in \mathcal{P}_m, ~\! g \in \mathcal{P}_n.
        \end{align*}
    Then for any $f \in \mathcal{P}_m, ~\! g \in \mathcal{P}_n$ and $h \in \mathcal{P}_k$, we have
        \begin{align} \label{preLie}
            \iota_h \iota_g f - \iota_{\iota_h g} f = (-1)^{(n-1)(k-1)} ~ \! (\iota_g \iota_h f - \iota_{\iota_g h} f).
        \end{align}
    Consequently, one obtains the following well-known graded Lie algebra structure associated with a nonsymmetric operad \cite{gers-voro}.
    
    \begin{theorem}
        Let $\mathcal{P}$ be a nonsymmetric operad. Then the graded vector space $\displaystyle \mathcal{P}_{\bullet +1}:= \bigoplus_{n \geq 0} \mathcal{P}_{n+1}$ equipped with the bracket 
            \begin{align*}
                [f,g]_{\mathsf{GV}}:= \iota_f g- (-1)^{mn}~ \! \iota_g f,
            \end{align*}
        for $f \in \mathcal{P}_{m+1}, ~ \! g \in \mathcal{P}_{n+1}$ forms a graded Lie algebra. 
    \end{theorem}

    The graded Lie algebra $(\mathcal{P}_{\bullet +1}, [~,~]_{\mathsf{GV}})$ is called the {\bf Gerstenhaber-Voronov algebra} associated with the nonsymmetric operad $\mathcal{P}$.
    When $\mathcal{P}$ is the endomorphism operad $\mathrm{End}_A$ associated with the vector space $A$, the Gerstenhaber-Voronov algebra associated with $\mathrm{End}_A$ is nothing but the Gerstenhaber's graded Lie algebra on the space of all multilinear maps on $A$ \cite{gers-ring}.

    \begin{definition} \cite{gers-voro}
        Let $\mathcal{P}$ be a nonsymmetric operad. An element $\pi \in \mathcal{P}_2$ is said to be a {\bf multiplication} on $\mathcal{P}$ if $\pi \circ_1 \pi = \pi \circ_2 \pi$ (equivalently, $\iota_\pi \pi =0).$
    \end{definition}
    Since the underlying field $\mathbf{k}$ has characteristics not equal to $2$, the condition in the above definition is equivalent to $[\pi, \pi]_{\mathsf{GV}}=0$. In other words, a multiplication $\pi$ is a Maurer-Cartan element in the Gerstenhaber-Voronov algebra $(\mathcal{P}_{\bullet +1}, [~,~]_{\mathsf{GV}})$ associated with the nonsymmetric operad $\mathcal{P}$. When $\mathcal{P}= \mathrm{End}_A$, a multiplication on $\mathcal{P}$ is an element $\pi \in (\mathrm{End}_A)_2 = \mathrm{Hom}(A^{\otimes 2},A)$ that satisfies $\pi \circ_1 \pi = \pi \circ_2 \pi$, where the partial compositions are given in (\ref{ass-circ}). Thus, if $\pi \in  \mathrm{Hom}(A^{\otimes 2},A)$ is graphically represented by the operation 
    \begin{tikzpicture}[scale=.25] \draw (0,0) -- (1,-1) node[below, right]{$\pi$};  \draw (1,-1) -- (2,0); \draw (1,-1) -- (1,-2);  \end{tikzpicture} having two inputs and one output then $\pi$ is a multiplication if and only if the following equality holds:
    \begin{center}
    \begin{tikzpicture}[scale=0.25]
        \draw (0,0) -- (2,-2) node[below, left]{$\pi$}; \draw (2,-2) -- (4,0) ; \draw (2,-2) -- (2,-4) ; \draw (2,0) -- (1,-1) node[below, left]{$\pi$}; \draw (4.65,-2) -- (5.35,-2); \draw (4.65, -2.35) -- (5.35, -2.35);
        \draw (6,0) -- (8,-2) node[below, right]{$\pi$};    \draw (8,-2) -- (10,0);     \draw (8,-2) -- (8,-4);     \draw (8,0) -- (9,-1) node[below, right]{$\pi$};
    \end{tikzpicture} .
   \end{center}
    In other words, any multiplication on the endomorphism operad $\mathrm{End}_A$ is equivalent to having an associative algebra structure on the vector space $A$.

    Like an associative algebra has a cohomology theory (the Hochschild cohomology) and a cup product operation on the cohomology, a multiplication on an arbitrary nonsymmetric operad also gives rise to similar structures. Let $\mathcal{P}$ be a nonsymmetric operad and $\pi$ be a multiplication on $\mathcal{P}$. For each $n \geq 1$, we define a map $\delta_{\pi} : \mathcal{P}_n \rightarrow \mathcal{P}_{n+1}$ by 
        \begin{align}\label{delta-pi}
            \delta_\pi (f) := - [\pi, f]_{\mathsf{GV}}, \text{ for } f \in \mathcal{P}_n.
        \end{align}
     Since $[\pi, \pi ]_{\mathsf{GV}}=0$, it turns out that $(\delta_\pi)^2 =0$. In other words, $\{ \mathcal{P}_\bullet , \delta_\pi  \}$ is a cochain complex. We define the corresponding cohomology groups by $H_\mathcal{P}^\bullet (\pi)$.

     \begin{remark}
    Let $\mathcal{P}$ be a nonsymmetric operad and $\pi$ be a multiplication on $\mathcal{P}$. A {\bf representation} of $\pi$ is a pair $(\pi^l, \pi^r)$ consisting of two elements $\pi^l, \pi^r \in \mathcal{P}_2$ that satisfy 
\begin{align*}
    \pi^l \circ_1 \pi = \pi^l \circ_2 \pi^l,  \qquad  \pi^r \circ_1 \pi^l = \pi^l \circ_2 \pi^r \quad  \text{ and } \quad \pi^r \circ_1 \pi^r = \pi^r \circ_2 \pi.
\end{align*}

%\textcolor{red}{It follows that the pair $(\pi^l, \pi^r)$ is a representation of the multiplication $\pi$, called the {\em adjoint representation}.}

\noindent Given a multiplication $\pi$ and a 
 representation $(\pi^l, \pi^r)$ of it, for each $n \geq 1$, one may define a map $\delta_{\pi; \pi^l, \pi^r} : \mathcal{P}_n \rightarrow \mathcal{P}_{n+1}$ by
\begin{align}\label{delta-pi-lr}
    \delta_{\pi; \pi^l, \pi^r}  (f) := (-1)^{n+1} ~\! \pi^r \circ_1 f + \pi^l \circ_2 f + \sum_{i=1}^n (-1)^i ~\! f \circ_i \pi, \text{ for } f \in \mathcal{P}_n.
\end{align}
Then it is easy to see that $(\delta_{\pi; \pi^l, \pi^r})^2 = 0$ and hence $\{ \mathcal{P}_\bullet , \delta_{\pi; \pi^l, \pi^r} \}$ is a cochain complex. The corresponding cohomology groups are denoted by $H^\bullet_\mathcal{P} (\pi; \pi^l, \pi^r)$ and are said to be the cohomologies associated with the multiplication $\pi$ with coefficients in the representation $(\pi^l, \pi^r)$. When $(\pi^l , \pi^r) = (\pi, \pi)$, the map $\delta_{\pi; \pi^l , \pi^r}$ coincides with the map $\delta_\pi$ given in (\ref{delta-pi}). Hence, in this case, the cohomology $H^\bullet_\mathcal{P} (\pi; \pi^l, \pi^r)$ coincides with $H^\bullet_\mathcal{P} (\pi)$.
\end{remark}

   Let $\mathcal{P}$ be a nonsymmetric operad and $\pi$ be a multiplication on $\mathcal{P}$. For any $m,n \geq 1$, one may also define a map (called the {\bf cup product}) $\cup_\pi: \mathcal{P}_m \times \mathcal{P}_n \rightarrow \mathcal{P}_{m+n}$ by 
        \begin{align}\label{cup-pro}
            f \cup_\pi g := (\pi \circ_2 g)\circ_1 f,  \text{ for } f \in \mathcal{P}_m, ~ \! g \in \mathcal{P}_n.
        \end{align}
    This cup product is graded associative in the sense that $(f \cup_\pi g)\cup_\pi h = f \cup_\pi (g \cup_\pi h)$, for $f \in \mathcal{P}_m, ~ \! g \in \mathcal{P}_n$ and $ h \in \mathcal{P}_k$.
    %As a consequence, the graded space $\mathcal{P}_\bullet = \bigoplus_{n \geq 1} \mathcal{P}_n$ equipped with the bracket 
       % \begin{align*}
       %     [f,g]_\pi:= f \cup_\pi g - (-1)^{mn}g \cup_\pi f ,
      %  \end{align*}
   % for $f \in \mathcal{P}_m , g \in \mathcal{P}_n$ is a graded Lie algebra. This is called the cup product graded Lie algebra.
 It is further easy to observe that the differential $\delta_\pi$ is a graded derivation for the cup product $\cup_\pi$. Therefore, the cup product $\cup_\pi$ induces a graded associative product on the graded space of cohomology $H_\mathcal{P}^\bullet (\pi)$.

\medskip

\section{Cup bracket}
Given a nonsymmetric operad with a multiplication $\pi$, here we consider the graded Lie bracket $[~,~]_\pi$ obtained by the skew-symmetrization of the cup-product operation (\ref{cup-pro}). We also consider the differential map $D$ associated with the multiplication $\pi$ with coefficients in the trivial representation. This differential map turns out to be a graded derivation for the bracket $[~,~]_\pi$ yielding a differential graded Lie algebra (dgLa). Finally, we use this dgLa to study maps that preserve a given multiplication.

\medskip

    Let $\mathcal{P}$ be a nonsymmetric operad and $\pi$ be a multiplication on $\mathcal{P}$. In the previous section, we have seen that $\pi$ induces a graded associative product $\cup_\pi$ on the graded space $\displaystyle \mathcal{P}_\bullet = \bigoplus_{n\geq 1} \mathcal{P}_n$. As a consequence, the graded space $\mathcal{P}_\bullet$ equipped with the bracket (called the {\bf cup bracket})
        \begin{align} \label{cup-pi}
            [f,g]_\pi :=  f \cup_\pi g- (-1)^{mn} ~ \! g \cup_\pi f, \text{ for } f \in \mathcal{P}_m, ~ \! g \in \mathcal{P}_n
        \end{align}
     forms a graded Lie algebra. This is called the {\bf cup Lie algebra}. In the following, we will formulate two other descriptions of the cup bracket (\ref{cup-pi}) that will be useful in forthcoming sections. As before, let $\mathcal{P}$ be a nonsymmetric operad and $\pi$ be a multiplication on $\mathcal{P}$. For any $n \geq 1$, we define a map $\theta : \mathcal{P}_n \rightarrow \mathcal{P}_{n+1}$ by 
        \begin{align} \label{theta}
            \theta (f):= - \iota_f \pi, \text{ for } f \in \mathcal{P}_n. 
        \end{align}
    In other words, for $f \in \mathcal{P}_n$, we have $ \theta(f) = -~\! \pi \circ_1 f + (-1)^{n} ~ \! \pi \circ_2 f$. With this notation, we have the following result.

    \begin{proposition} \label{pi-bracket}
        For any $f \in \mathcal{P}_m$ and $g \in \mathcal{P}_n$, we have
            \begin{align}
                [f,g]_\pi &= (-1)^n ~ \! \big( \iota_f (\theta g)- \theta (\iota_f g)   \big), \label{first}\\
                [f,g]_\pi &= \iota_f (\delta_\pi g)+ (-1)^{m-1} ~ \!  \iota_{\delta_\pi f} g + (-1)^m ~ \! \delta_\pi (\iota_f g). \label{second}
            \end{align}
    \end{proposition}
    \begin{proof}
        We observe that
        \begin{align*}
            \iota_f (\theta g) &= \sum_{i=1}^{n+1} (-1)^{(i-1)(m-1)}~ \! (\theta g ) \circ_i f\\
            &= - \sum_{i=1}^{n+1}(-1)^{(i-1)(m-1)}~\!(\pi \circ_1 g)\circ_i f +(-1)^{n}~\!\sum_{i=1}^{n+1}(-1)^{(i-1)(m-1)}~\!(\pi \circ_2 g)\circ_i f\\
            &= - \sum_{i=1}^{n}(-1)^{(i-1)(m-1)} ~\! (\pi \circ_1 g)\circ_i f- (-1)^{n(m-1)} ~ \! (\pi \circ_1 g)\circ_{n+1}f
              + (-1)^{n} ~ \! (\pi \circ_2 g)\circ_1 f \\ & \quad  + (-1)^{n}\sum_{i=2}^{n+1}(-1)^{(i-1)(m-1)} ~ \! (\pi \circ_2 g)\circ_i f\\
            &= -\sum_{i=1}^{n}(-1)^{(i-1)(m-1)} ~ \! \pi \circ_1 (g\circ_i f) - (-1)^{n(m-1)} ~ \! (\pi \circ_2 f)\circ_1 g + (-1)^{n} ~ \! (\pi \circ_2 g)\circ_1 f \\  & \quad  + (-1)^{n}\sum_{i=1}^{n}(-1)^{i(m-1)} ~ \! (\pi \circ_2 g)\circ_{i+1} f\\
            &= -\bigg( \sum_{i=1}^{n}(-1)^{(i-1)(m-1)} ~ \! \pi \circ_1 (g\circ_i f) - (-1)^{n}\sum_{i=1}^{n}(-1)^{i(m-1)} ~ \! \pi \circ_2 (g\circ_{i} f) \bigg) +(-1)^n~\!  (\pi \circ_2 g)\circ_1 f\\ & \quad 
              - (-1)^n (-1)^{mn} ~ \! (\pi \circ_2 f)\circ_1 g \\
            &= -\big( \pi \circ_1 \iota_f g + (-1)^{m+n-2} ~ \!  \pi \circ_2 \iota_f g   \big) + (-1)^n ~ \! [f,g]_\pi
            = \theta (\iota_f g) + (-1)^n ~ \! [f,g]_\pi.
        \end{align*} 
        Hence, the identity (\ref{first}) follows. On the other hand, we have
        \begin{align*}
            &\iota_f (\delta_\pi g)+ (-1)^{m-1} ~ \! \iota_{\delta_\pi f} g + (-1)^m ~ \! \delta_\pi (\iota_f g)\\
            &= -\iota_f [\pi, g]_\mathsf{GV} - (-1)^{m-1} ~ \! \iota_{[\pi, f]_\mathsf{GV}} g + (-1)^m ~ \! [\pi , \iota_f g]_\mathsf{GV} \\
            &= - \iota_f \iota_\pi g + (-1)^{n-1}~ \! \iota_f \iota_g \pi - (-1)^{m-1} ~ \! \iota_{\iota_\pi f}g + \iota_{\iota_f \pi} g -(-1)^m ~ \! \iota_\pi \iota_f g+ (-1)^m (-1)^{m+n-2} ~ \! \iota_{\iota_f g}\pi \\
            &= (-1)^n ~ \! (-\iota_f \iota_g \pi + \iota_{\iota_f g}\pi) + \Big(\iota_{\iota_f \pi}g - \iota_f \iota_\pi g -(-1)^{m-1} ~ \! (\iota_{\iota_\pi f} g - \iota_\pi \iota_f g)\Big)\\
            &\stackrel{(\ref{preLie})}{=} (-1)^n ~ \! (\iota_f \theta g - \theta (\iota_f g)) = [f,g]_\pi.
        \end{align*}
        This proves the identity (\ref{second}) and completes the proof.
    \end{proof}

    \begin{remark}
        Note that the differential $\delta_\pi$ is a graded derivation for the cup product $\cup_\pi$, which in turn implies that $\delta_\pi$ is also a graded derivation for the cup bracket $[~,~]_\pi$. That is, $(\mathcal{P}_\bullet , [~,~]_\pi , \delta_\pi)$ is a differential graded Lie algebra. This, in particular, implies that the cup bracket $[~,~]_\pi$ induces a graded Lie bracket on the graded space of cohomology $H_\pi ^\bullet (\mathcal{P})$. However, it follows from the identity (\ref{second}) that the induced graded Lie bracket at the level of cohomology is trivial. Equivalently, it means that the induced cup product at the level of cohomology is graded commutative.
    \end{remark}

    Given a multiplication $\pi$, we can also define a map $D :\mathcal{P}_n \rightarrow \mathcal{P}_{n+1}  $ (for $n \geq 1$) by 
        \begin{align} \label{derivation-map}
            D(f) = - \iota_\pi f,  \text{ for } f \in \mathcal{P}_n.
        \end{align}
        It follows from (\ref{delta-pi-lr}) that the map $D$ is the coboundary operator associated to the multiplication $\pi$ with coefficients in the trivial representation $(\pi^l , \pi^r) = (0, 0)$. Hence, $D$ is a differential. Alternatively, we can directly see that $ D^2(f)= \iota_\pi \iota_\pi f \stackrel{(\ref{preLie})}{=} \iota_{\iota_\pi \pi} f = 0 $. Moreover, we have the following result.

    \begin{proposition}
        The map $D$ is a graded derivation for the cup bracket, i.e., 
            \begin{align*}
                D[f,g]_\pi = [D(f),g]_\pi + (-1)^m ~ \! [f, D(g)]_\pi , ~ \text{for } f \in \mathcal{P}_m, ~ \! g \in \mathcal{P}_n.
            \end{align*}
    \end{proposition}
    \begin{proof}
        %We start by expanding both the terms on RHS by direct calculation. This gives us
        %Expanding the bracket $[~,~]_\pi$ and the map $D$ on the RHS of the above expression, we get
        %\begin{align} \label{rhs}
            %& [D(f),g]_\pi + (-1)^m ~ \!  [f,D(g)]_\pi  \nonumber \\
            %& = -\iota_\pi f \cup_\pi g + (-1)^{n(m+1)} ~ \!  g \cup_\pi \iota_\pi f + (-1)^m ~ \! \big( - f \cup_\pi \iota_\pi g + (-1)^{m(n+1)} ~ \! \iota_\pi g \cup_\pi f \big) \nonumber \\
            %&= -(\pi \circ_2 g) \circ_1 \iota_\pi f + (-1)^{n(m+1)} ~ \! (\pi \circ_2 \iota_\pi f) \circ_1 g + (-1)^m ~ \! \big( -(\pi \circ_2 \iota_\pi g)\circ_1f +(-1)^{m(n+1)} ~ \! (\pi \circ_2 f)\circ_1 \iota_\pi g \big) \nonumber \\
            %&= - \sum_{i=1}^m (-1)^{i-1} ~ \! (\pi \circ_2 g) \circ_1 (f \circ_i \pi) + (-1)^{n(m+1)} ~ \! \sum_{i=1}^m (-1)^{i-1} (\pi \circ_2 (f \circ_i \pi))\circ_1 g \nonumber \\
           % &\qquad + (-1)^m \Big( - \sum_{i=1}^n (-1)^{i-1}(\pi \circ_2 (g \circ_i \pi))\circ_1 f + (-1)^{m(n+1)}\sum_{i=1}^n (-1)^{i-1} (\pi \circ_2 f)\circ_1 (g \circ_i \pi) \Big).
       % \end{align}
    Using the definitions of the map $D$ and the bracket $[~,~]_\pi$, we expand the LHS of the above expression. This gives us 
        \begin{align*} %\label{lhs}
            & D[f,g]_\pi = -\iota_\pi [f,g]_\pi 
            = -\iota_\pi (f \cup_\pi g) + (-1)^{mn} ~ \! \iota_\pi (g \cup_\pi f) \nonumber \\
            &= -\sum_{i=1}^{m+n} (-1)^{i-1} ~ \! ((\pi \circ_2 g) \circ_1 f) \circ_i \pi + (-1)^{mn}~\! \sum_{i=1}^{m+n} (-1)^{i-1} ~ \! ((\pi \circ_2 f) \circ_1 g) \circ_i \pi  \nonumber \\
            &= -\sum_{i=1}^{m} (-1)^{i-1} ~ \! ((\pi \circ_2 g) \circ_1 f) \circ_i \pi -\sum_{i=m+1}^{m+n} (-1)^{i-1} ~ \! ((\pi \circ_2 g) \circ_1 f) \circ_i \pi \nonumber\\
            &\quad + (-1)^{mn}~\! \sum_{i=1}^{n} (-1)^{i-1} ~ \! ((\pi \circ_2 f) \circ_1 g) \circ_i \pi +(-1)^{mn} \sum_{i=n+1}^{n+m} (-1)^{i-1} ~ \! ((\pi \circ_2 f) \circ_1 g) \circ_i \pi  \nonumber \\
            &= -\sum_{i=1}^{m} (-1)^{i-1} ~ \! ((\pi \circ_2 g) \circ_1 f) \circ_i \pi - (-1)^m ~\!\sum_{i=1}^{n} (-1)^{i-1} ~ \! ((\pi \circ_2 g) \circ_1 f) \circ_{m+i} \pi \nonumber \\
            & \quad + (-1)^{mn}~\! \sum_{i=1}^{n} (-1)^{i-1} ~ \! ((\pi \circ_2 f) \circ_1 g) \circ_i \pi +(-1)^{(m+1)n}~\! \sum_{i=1}^{m} (-1)^{i-1} ~ \!  ((\pi \circ_2 f) \circ_1 g) \circ_{n+i} \pi \nonumber\\
            &= - \sum_{i=1}^m (-1)^{i-1} ~ \! (\pi \circ_2 g) \circ_1 (f \circ_i \pi) -(-1)^m ~\!\sum_{i=1}^n (-1)^{i-1} ~ \! (\pi \circ_2 (g \circ_i \pi))\circ_1 f \nonumber \\ & \quad +(-1)^{mn}~\! \sum_{i=1}^{n} (-1)^{i-1} ~ \! (\pi \circ_2 f)\circ_1 (g \circ_i \pi) +(-1)^{(m+1)n}~\! \sum_{i=1}^{m} (-1)^{i-1} ~ \!  (\pi \circ_2 (f \circ_i \pi))\circ_1 g \nonumber \\
            &= -(\pi \circ_2 g) \circ_1 \iota_\pi f - (-1)^m ~\!(\pi \circ_2 \iota_\pi g)\circ_1f + (-1)^{mn}~\! (\pi \circ_2 f) \circ_1 \iota_\pi g +(-1)^{(m+1)n}~\! (\pi \circ_2 \iota_\pi f) \circ_1 g \nonumber \\
            & = -\iota_\pi f \cup_\pi g + (-1)^{n(m+1)} ~ \!  g \cup_\pi \iota_\pi f + (-1)^m ~ \! \big( - f \cup_\pi \iota_\pi g + (-1)^{m(n+1)} ~ \! \iota_\pi g \cup_\pi f \big) \nonumber \\
            &= [D(f),g]_\pi + (-1)^m ~ \!  [f,D(g)]_\pi.
           \end{align*}
%For any $1 \leq i \leq m$, we have $(\pi \circ_2 g) \circ_1 (f \circ_i \pi)=((\pi \circ_2 g) \circ_1 f) \circ_i \pi $. Hence the first summand of (\ref{rhs}) and (\ref{lhs}) are equal. We also have $(\pi \circ_2 (f \circ_i \pi))\circ_1 g = ((\pi \circ_2 f) \circ_1 g) \circ_{n+i} \pi$ for any $1 \leq i \leq m$, implying the second term of (\ref{rhs}) is equal to the last term of (\ref{lhs}). Similarly, for $1 \leq i \leq n$, we have the equalities 
        %\begin{align*}
           % (\pi \circ_2 (g \circ_i \pi))\circ_1 f= ((\pi \circ_2 g) \circ_1 f) \circ_{m+i} \pi \quad  \text{ and } \quad  (\pi \circ_2 f) \circ_1 (g \circ_i \pi)=((\pi \circ_2 f) \circ_1 g) \circ_i \pi.
        %\end{align*}
        %Hence last two terms of (\ref{rhs}) are the same as the remaining two terms of (\ref{lhs}). This completes the proof. 
        Hence the proof.
        \end{proof}

    The above proposition shows that the triple $(\mathcal{P}_\bullet, [~,~]_\pi, D)$ is also a differential graded Lie algebra.

    \begin{remark}
        From the definitions of the maps $\delta_\pi, ~ \! \theta$ and $D$, it is clear that
        \begin{align} \label{delta-D-theta}
            \delta_\pi (f) = D(f) +(-1)^n ~ \!  \theta (f),  \text{ for } f \in \mathcal{P}_n.
        \end{align}
         Since $\delta_\pi$ and $D$ are both graded derivations for the bracket $[~,~]_\pi$, the map $\theta$ satisfies the following derivation-like property:
            \begin{align} \label{theta-pi}
                \theta [f,g]_\pi = (-1)^n ~ \! [\theta f,g]_\pi + [f, \theta g]_\pi,  \text{ for } f \in \mathcal{P}_m, ~ \! g \in \mathcal{P}_n.
            \end{align}
    \end{remark}

    \begin{definition}\label{preserve}
        Let $\mathcal{P}$ be a nonsymmetric operad and $\pi, ~\! \pi'$ be two multiplications on $\mathcal{P}$. A {\bf map} from $\pi$ to $\pi'$ is an element $\varphi \in \mathcal{P}_1$ satisfying $ ~\varphi \circ_1 \pi = (\pi'\circ_2 \varphi) \circ_1 \varphi$. An element $\varphi \in \mathcal{P}_1$ is said to preserve the multiplication $\pi$ if $\varphi$ is a map from $\pi$ to itself.
    \end{definition}
    
%    \begin{definition} 
  %      Let $\mathcal{P}$ be a nonsymmetric operad and $\pi$ be a multiplication on $\mathcal{P}$. An element $\varphi \in \mathcal{P}_1$ is said to preserve the multiplication $\pi$ if 
 %           \begin{align*}
  %              \varphi \circ_1 \pi = (\pi \circ_2 \varphi) \circ_1 \varphi.
  %          \end{align*}
   % \end{definition}

    Let $\mathcal{P}= \mathrm{End}_A$ be the endomorphism operad associated to the vector space $A$ and let $\pi,~ \! \pi'$ be two multiplications on $\mathcal{P}$. That is, they define associative algebra structures on the space $A$. Note that an element $\varphi \in (\mathrm{End}_A)_1 = \mathrm{Hom} (A, A)$ is a map from $\pi$ to $\pi'$ in the sense of Definition \ref{preserve} if and only if $\varphi: A \rightarrow A$ is an associative algebra homomorphism from $(A, \pi)$ to $(A, \pi')$.

    \begin{proposition} \label{MC-multiplication}
        Let $\mathcal{P}$ be a nonsymmetric operad and $\pi$ be a multiplication on it. Then an element $\varphi \in \mathcal{P}_1$ preserves the multiplication $\pi$ if and only if $\varphi$ is a Maurer-Cartan element in the differential graded Lie algebra $(\mathcal{P}_\bullet , [~,~]_\pi, D)$.
    \end{proposition}
    \begin{proof}
        We observe that
        \begin{align*}
            D(\varphi) + \frac{1}{2}[\varphi, \varphi]_\pi
            = -\iota_\pi \varphi + \frac{1}{2}(\varphi \cup_\pi \varphi + \varphi \cup_\pi \varphi)
            = - \varphi \circ_1 \pi + (\pi \circ_2 \varphi)\circ_1 \varphi.
        \end{align*}
        This shows that $\varphi$ is a Maurer-Cartan element  (i.e., $D(\varphi) + \frac{1}{2}[\varphi, \varphi]_\pi=0$) if and only if $\varphi$ preserves the multiplication $\pi$.
    \end{proof}

    Let $\varphi \in \mathcal{P}_1$ be an element that preserves the multiplication $\pi$. It follows from the above proposition that $\varphi$ induces a differential $D_\varphi :\mathcal{P}_n \rightarrow \mathcal{P}_{n+1}$ (for $n\geq 1$) given by 
        \begin{align*}
            D_\varphi(f) = D(f) + [\varphi,f]_\pi, \text{ for } f \in \mathcal{P}_n.
        \end{align*}

    Then we have the following result.
    \begin{proposition}
        With the above notations, the triple $(\mathcal{P}_\bullet , [~,~]_\pi , D_\varphi)$ is also a differential graded Lie algebra. Moreover, for any element $\varphi' \in \mathcal{P}_1$, the sum $\varphi + \varphi'$ preserves the multiplication $\pi$ if and only if $\varphi'$ is a Maurer-Cartan element in the differential graded Lie algebra $(\mathcal{P}_\bullet, [~,~]_\pi, D_\varphi)$.
    \end{proposition}
    \begin{proof}
        For any $f\in \mathcal{P}_m$ and $g \in \mathcal{P}_n$, we have
        \begin{align*}
            D_\varphi^2(f)&= D_\varphi (D(f) + [\varphi , f]_\pi) \nonumber\\
            &= D^2(f) + D[\varphi, f]_\pi + [\varphi, D(f)]_\pi + [\varphi , [\varphi, f]_\pi]_\pi\nonumber\\
            &= [D(\varphi),f]_\pi + \frac{1}{2}[[\varphi, \varphi]_\pi, f]_\pi \quad (\text{\small by graded derivation of $D$ and graded Jacobi identity of $[~,~]_\pi$})\nonumber\\
            &= [D(\varphi)+\frac{1}{2}[\varphi, \varphi]_\pi, f]_\pi 
            = 0 \quad (\text{using Proposition \ref{MC-multiplication}})
        \end{align*}
        and
        \begin{align*} 
            &D_\varphi [f,g]_\pi - [D_\varphi (f), g]_\pi - (-1)^m ~ \!  [f, D_\varphi (g)]_\pi \nonumber\\
            &= D[f,g]_\pi +[\varphi, [f,g]_\pi]_\pi - [D(f),g]_\pi- [[\varphi,f]_\pi , g]_\pi -(-1)^m ~ \! [f,D(g)]_\pi - (-1)^m ~ \! [f,[\varphi, g]_\pi]_\pi \nonumber \\
            &= [\varphi, [f,g]_\pi]_\pi - [[\varphi, f]_\pi, g]_\pi - (-1)^m ~ \! [f,[\varphi, g]_\pi]_\pi = 0 \quad (\text{by graded Jacobi identity of $[~,~]_\pi$}).
        \end{align*}
        This shows that $(\mathcal{P}_\bullet , [~,~]_\pi , D_\varphi)$ is a differential graded Lie algebra.
For the second part, we observe that 
        \begin{align*}
        D (\varphi + \varphi') + \frac{1}{2} [ \varphi + \varphi' ~ \! , ~ \! \varphi + \varphi']_\pi
            =&~ D (\varphi) +D (\varphi' ) + \frac{1}{2}\big( 
 [\varphi, \varphi]_\pi + 2 ~ \! [\varphi, \varphi']_\pi + [\varphi', \varphi']_\pi \big)\\
            =&~  D (\varphi') + [\varphi, \varphi']_\pi +\frac{1}{2}[\varphi',\varphi']_\pi\\
            =&~ D_\varphi (\varphi') + \frac{1}{2} [\varphi', \varphi']_\pi.
        \end{align*}
        This shows that $\varphi + \varphi'$ is a Maurer-Cartan element in the differential graded Lie algebra $(\mathcal{P}_\bullet , [~,~]_\pi, D)$ if and only if $\varphi'$ is a Maurer-Cartan element in $(\mathcal{P}_\bullet , [~,~]_\pi, D_\varphi)$. Hence, the result follows.
    \end{proof}

\begin{remark}
    Let $\varphi \in \mathcal{P}_1$ preserves the multiplication $\pi$. Then we have seen that $D_\varphi$ is a differential, i.e., \{$\mathcal{P}_\bullet, D_\varphi$\} is a cochain complex. The corresponding cohomology groups are denoted by $H^\bullet_{\mathcal{P}, \pi} (\varphi )$, and are called the cohomology groups of $\varphi$. On the other hand, we can define two elements $\pi^l , \pi^r \in \mathcal{P}_2$ by $\pi^l := \pi \circ_1 \varphi$  and $\pi^r := \pi \circ_2 \varphi$, respectively. Then it is easy to see that the pair $(\pi^l, \pi^r)$ is a representation of the multiplication $\pi$. Moreover, the coboundary map associated with the multiplication $\pi$ with coefficients in the above-representation $(\pi^l, \pi^r)$ coincides with the map $D_\varphi$. Hence, the cohomology of $\varphi$ can be viewed as the cohomology of the multiplication $\pi$ with coefficients in the induced representation $(\pi^l, \pi^r)$.
    %Let $\pi$ be a multiplication in the endomorphism operad $\mathrm{End}_A$ and $\varphi \in (\mathrm{End}_A)_1 = \mathrm{Hom} (A,A)$ preserves the multiplication $\pi$ (i.e., $\varphi$ is an associative algebra homomorphism from $(A, \pi)$ to itself). In this case, the differential map $D_\varphi$ is simply the Hochschild coboundary map of the associative algebra $A_\pi = (A, \pi)$ with coefficients in the $A_\pi$-bimodule $A$, where the left and right $A_\pi$-module structures on $A$ are respectively given by
    %\begin{align*}
     %   A_\pi \otimes A \xrightarrow{ \pi \circ_1 \varphi} A \quad \text{ and } \quad A \otimes A_\pi \xrightarrow{\pi \circ_2 \varphi} A.
    %\end{align*}
    %The corresponding cohomology groups were considered by Gerstenhaber and Schack \cite{gers-sch} when they studied the formal deformations of the algebra homomorphism $\varphi$.
    \end{remark}

\medskip

\section{Nijenhuis elements and the Fr\"olicher-Nijenhuis bracket}
Given a nonsymmetric operad with a multiplication $\pi$, here we first consider Nijenhuis elements for the multiplication $\pi$ as a generalization of Nijenhuis operators on associative algebras \cite{grab}.
We show that a Nijenhuis element induces a hierarchy of pairwise compatible Nijenhuis elements generalizing a well-known result about classical Nijenhuis operators. Next, we define the Fr\"{o}licher-Nijenhuis bracket associated with a given multiplication $\pi$ in a nonsymmetric operad. The Maurer-Cartan elements of the Fr\"{o}licher-Nijenhuis bracket are precisely the Nijenhuis elements for the multiplication $\pi$. As a result, we define the cohomology associated with a Nijenhuis element.

    \begin{definition}
        Let $\mathcal{P}$ be a nonsymmetric operad and $\pi$ be a multiplication on $\mathcal{P}$. An element $N \in \mathcal{P}_1$ is said to be a {\bf Nijenhuis element} for the multiplication $\pi$ if 
            \begin{align}\label{nij-ele-iden}
                (\pi \circ_2 N)\circ_1 N = N \circ_1 (\pi \circ_1 N + \pi \circ_2 N - N \circ_1 \pi).
            \end{align}
    \end{definition}

    \begin{remark}
        Let $\mathcal{P}= \mathrm{End}_A$ be the endomorphism operad associated with the vector space $A$. Also, let $\pi \in \mathcal{P}_2 = (\mathrm{End}_A)_2$ be a multiplication on $\mathcal{P}$ representing the associative algebra structure $(A, ~\! \cdot ~\!)$ on the space $A$, where $\pi (a, b) = a \cdot b$, for all $a, b \in A$. Then an element $N \in \mathcal{P}_1 = (\mathrm{End}_A)_1 = \mathrm{Hom}(A, A)$ is a Nijenhuis element for the multiplication $\pi$ if and only if $N$ is a Nijenhuis operator on the associative algebra $(A, ~\! \cdot ~\! )$ \cite{grab}, i.e.,
        \begin{align*}
            N (a) \cdot N (b) = N ( N (a) \cdot b + a \cdot N (b) - N (a \cdot b)), \text{ for all } a, b \in A.
        \end{align*}
        Thus, a Nijenhuis element generalizes classical Nijenhuis operators on associative algebras.
       % Then an element $N\in (\mathrm{End}_A )_1= \mathrm{Hom}(A, A)$ is a Nijenhuis element for a multiplication $\pi$ if and only if the linear map $N: A \ A\rightarrow A$ is a Nijenhuis operator on the associative algebra $(A, \pi)$. 
    \end{remark}

Let $N$ be a Nijenhuis element for the multiplication $\pi$. For any $k \geq 1$, we set
\begin{align*}
    N^k := N \circ_1 N^{k-1} ~~~ (\text{where by convention } N^0 = \mathbbm{1}) ~~~ \text{ and } ~~~~ \pi_{N^k} := \pi \circ_1 N^k + \pi \circ_2 N^k - N^k \circ_1 \pi. 
\end{align*}

\begin{proposition}
    Let $N$ be a Nijenhuis element for the multiplication $\pi$. Then for any $k, l \geq 0$,
    \begin{align}\label{nkl}
     ( \pi_{N^l} \circ_2 N^k ) \circ_1 N^k =   N^k \circ_1 \pi_{N^{k+l}}.
    \end{align}
\end{proposition}

\begin{proof}
    Note that the identity (\ref{nkl}) trivially holds for $k = 0$ and any $l \geq 0$. In the following, we will first prove it for $k=1$. Since $N$ is a Nijenhuis element for the multiplication $\pi$, the identity (\ref{nij-ele-iden}) holds. By inserting both sides of this identity inside $N^l$, we obtain
    \begin{align}\label{iden-22}
        N^{l+2} \circ_1 \pi - N^{l+1} \circ_1 (\pi \circ_2 N) = N^{l+1} \circ_1 (\pi \circ_1 N) - N^l \circ_1 ((\pi \circ_2 N) \circ_1 N).
    \end{align}
    Applying this identity repeatedly, we get
    \begin{align}\label{iden-23}
     N^{l+2} \circ_1 \pi - N^{l+1} \circ_1 (\pi \circ_2 N) = N \circ_1 (\pi \circ_1 N^{l+1}) - (\pi \circ_2 N) \circ_1 N^{l+1}.
    \end{align}
    Similarly,
    \begin{align}\label{iden-24}
        N^{l+2} \circ_1 \pi - N^{l+1} \circ_1 (\pi \circ_1 N) = N \circ_1 (\pi \circ_2 N^{l+1}) - (\pi \circ_2 N^{l+1}) \circ_1 N.
    \end{align}
    Hence comparing the identities (\ref{iden-22}) and (\ref{iden-24}), we get that 
    \begin{align}\label{iden-25}
        N^{l+1} \circ_1 (\pi \circ_2 N) - N^l \circ_1 ((\pi \circ_2 N) \circ_1 N) = N \circ_1 (\pi \circ_2 N^{l+1}) - (\pi \circ_2 N^{l+1}) \circ_1 N.
    \end{align}
    Again, comparing the element $N^{l+1} \circ_1 (\pi \circ_2 N)$ from the identities (\ref{iden-23}) and (\ref{iden-25}), we find
    \begin{align*}
        &N^{l+2} \circ_1 \pi - N \circ_1 (\pi \circ_1 N^{l+1}) + (\pi \circ_2 N) \circ_1 N^{l+1} \\
         & \qquad \qquad \qquad \qquad = N^l \circ_1 ( (\pi \circ_2 N) \circ_1 N) + N \circ_1 (\pi \circ_2 N^{l+1}) - (\pi \circ_2 N^{l+1}) \circ_1 N. 
    \end{align*}
    This can be equivalently expressed as 
    \begin{align} \label{n1l}
        (\pi_{N^l} \circ_2 N) \circ_1 N = N \circ_1 ( \pi \circ_1 N^{l+1} + \pi \circ_2 N^{l+1} - N^{l+1} \circ_1 \pi) 
    \end{align}
    which proves the identity (\ref{nkl}) for $k = 1$. Hence for arbitrary $k \geq 1$, 
    \begin{align*}
        N^k \circ_1 \pi_{N^{k+l}} = N^{k-1} \circ_1 (N \circ_1 \pi_{N^{k+l}}) =~& N^{k-1} \circ_1 (( \pi_{N^{k+l-1}} \circ_2 N) \circ_1 N ) \quad (\text{by } (\ref{n1l})) \\
        =~& (\pi_{N^l} \circ_2 N^k) \circ_1 N^k  \quad (\text{by repeating } (\ref{n1l})) .
    \end{align*}
    This proves the result.
\end{proof}

\begin{theorem}
    Let $N$ be a Nijenhuis element for the multiplication $\pi$.
    \begin{enumerate}
        \item For any $k \geq 0$, the element $N^k$ is a Nijenhuis element for the multiplication $\pi$.
        \item For any $k, l \geq 0$, the element $N^k$ is a Nijenhuis element for the deformed multiplication $\pi_{N^l}$. Moreover, the multiplications $(\pi_{N^l})_{N^k}$ and $\pi_{N^{k+l}}$ are the same.
        \item For any $k, l \geq 0$, the Nijenhuis elements $N^k$ and $N^l$ (for the multiplication $\pi$) are `compatible' in the sense that their sum is also a Nijenhuis element for the multiplication $\pi$.
    \end{enumerate}
\end{theorem}
    
\begin{proof}
    (1) It follows from the identity (\ref{nkl}) by taking $l = 0$.

    \medskip

    (2) To prove this part, we first claim that $(\pi_{N^l})_N  = \pi_{N^{l+1}}$, for any $l \geq 0$. Note that it trivially holds for $l=0$. However, for $l \geq 1$, we observe that
    \begin{align*}
        &(\pi_{N^l})_N =\pi_{N^l} \circ_1 N + \pi_{N^l} \circ_2 N - N \circ_1 \pi_{N^l} \\
        &= (\pi \circ_1 N^l + \pi \circ_2 N^l - N^l \circ_1 \pi) \circ_1 N +  (\pi \circ_1 N^l + \pi \circ_2 N^l - N^l \circ_1 \pi) \circ_2 N - (\pi_{N^{l-1}} \circ_2 N) \circ_1 N \\
        &= \pi \circ_1 N^{l+1} + (\pi \circ_2 N^l) \circ_1 N - N^l \circ_1 (\pi \circ_1 N) + (\pi \circ_2 N) \circ_1 N^l + \pi \circ_2 N^{l+1} - N^l \circ_1 (\pi \circ_2 N) \\
        & \qquad \qquad - (\pi_{N^{l-1}} \circ_2 N) \circ_1 N \\
        &= \big(  \pi \circ_1 N^{l+1} + \pi \circ_2 N^{l+1} - N^{l+1} \circ_1 \pi \big) - N^l \circ_1 \big(  \pi \circ_1 N + \pi \circ_2 N - N \circ_1 \pi   \big) + N^{l-1} \circ_1 ( (\pi \circ_2 N) \circ_1 N) \\
        &= \pi_{N^{l+1}} - N^{l-1} \circ_1 \big(\underbrace{ N \circ_1 \pi_N - (\pi \circ_2 N) \circ_1 N  }_{= 0}  \big) = \pi_{N^{l+1}}.
    \end{align*}
Using \eqref{nkl}, we get
    \begin{align*}
        (\pi_{N^l} \circ_2 N) \circ_1 N = N \circ_1 \pi_{N^{l+1}}= N \circ_1 (\pi_{N^l})_N,
    \end{align*}
    which proves that $N$ is a Nijenhuis element for the multiplication $\pi_{N^l}$. Replacing $N$ by the Nijenhuis element $N^k$ in the above identity, we get that $N^k$ is also Nijenhuis element for the multiplication $\pi_{N^l}$ and $(\pi_{N^l})_{N^k}=\pi_{N^{k+l}}$.
%Hence, we have
%\begin{align*}
    %(\pi_{N^l} \circ_2 N^k) \circ_1 N^k = N^k \circ_1 \pi_{N^{k+l}} = N^k \circ_1 (\pi_{N^l} \circ_1 N^k + \pi_{N^l} \circ_2 N^k - N^k \circ_1 \pi_{N^l})
%\end{align*}
%which shows that $N^k$ is a Nijenhuis element for the multiplication $\pi_{N^l}$. Finally, the elements $(\pi_{N^l})_{N^k}$ and $\pi_{N^{k+l}}$ are the same because of (\textcolor{red}{Left}).
\medskip

    (3) For any $k \geq l$, we observe that
    \begin{align*}
        &N^k \circ_1 \pi_{N^l} + N^l \circ_1 \pi_{N^k} \\
        &= N^{k-l} \circ_1 ( (\pi \circ_2 N^l) \circ_1 N^l) +  (\pi_{N^{k-l}} \circ_2 N^l) \circ_1 N^l \\
        &= N^{k-l} \circ_1 ( (\pi \circ_2 N^l) \circ_1 N^l) + \big( ( \pi \circ_1 N^{k-l} + \pi \circ_2 N^{k-l} - N^{k-l} \circ_1 \pi) \circ_2 N^l  \big) \circ_1 N^l\\
        &= N^{k-l} \circ_1 ( (\pi \circ_2 N^l) \circ_1 N^l) + (\pi \circ_2 N^l) \circ_1 N^k + (\pi \circ_2 N^k) \circ_1 N^l - (N^{k-l} \circ_1 (\pi \circ_2 N^l)) \circ_1 N^l\\
        &= (\pi \circ_2 N^l) \circ_1 N^k + (\pi \circ_2 N^k) \circ_1 N^l.
    \end{align*}
    This is equivalent to the condition that the sum $N^k + N^l$ is a Nijenhuis element for the multiplication $\pi$.
\end{proof}

     In the following, we construct a graded Lie algebra whose Maurer-Cartan elements are Nijenhuis elements for a given multiplication $\pi$. We start with the following important result.

    \begin{theorem} \label{rho-theorem}
        Let $\mathcal{P}$ be a nonsymmetric operad and $\pi$ be a multiplication. For each $m \geq 0$ and $n \geq 1$, we define a map
        \begin{align} \label{rho-map}
            \rho : \mathcal{P}_{m+1} \times \mathcal{P}_n \rightarrow \mathcal{P}_{m+n} ~~ \text{ by } ~~   (f,\phi) \mapsto \rho(f)\phi := \iota_f \phi,
        \end{align}
        for $f \in \mathcal{P}_{m+1}$ and $\phi \in \mathcal{P}_n$. Then the map $\rho$ satisfies the following two conditions:
            \begin{align}
                \rho ([f,g]_{\mathsf{GV}}) (\phi) &= \rho (f) \rho (g) \phi - (-1)^{mn} ~ \! \rho (g) \rho (f) \phi, \label{action_1}\\
                \rho(f) [\phi , \psi]_{\pi} &= [ \rho (f) \phi , \psi]_\pi + (-1)^{mk}~ \! [\phi , \rho (f) \psi]_\pi \label{action_2},
            \end{align}
        for the elements $f \in \mathcal{P}_{m+1}, ~ \! g \in \mathcal{P}_{n+1},~ \! \phi \in \mathcal{P}_k$ and $\psi \in \mathcal{P}_l$.
    \end{theorem}
    \begin{proof}
        We have
        \begin{align*}
            \rho ([f,g]_{\mathsf{GV}}) (\phi) &= \iota_{[f,g]_\mathsf{GV}}\phi
            = \iota_{\iota_f g} \phi - (-1)^{mn} ~ \!  \iota_{\iota_g f} \phi
            \stackrel{(\ref{preLie})}{=} \iota_f \iota_g \phi - (-1)^{mn} ~ \! \iota_f \iota_g \phi\\
            &= \rho (f) \rho (g) \phi - (-1)^{mn}\rho (g) \rho (f) \phi.
        \end{align*}
        Hence, the identity (\ref{action_1}) follows. To prove (\ref{action_2}), we start by expanding the LHS by direct calculations
        %\begin{align} \label{rhs1}
           % &[\rho(f) \phi , \psi]_\pi + (-1)^{mk} [\phi , \rho(f) \psi]_\pi \\
           % &= [\iota_f \phi , \psi]_\pi (-1)^{mk} [\phi, \iota_f \psi]_\pi  \nonumber\\
            %&= \iota_f \phi \cup_\pi \psi - (-1)^{l(m+k)}\psi \cup_\pi \iota_f \phi +(-1)^{mk} \big( \phi \cup_\pi \iota_f \psi - (-1)^{k(m+l)} \iota_f \psi \cup_\pi \phi \big) \\
            %&= (\pi \circ_2 \psi) \circ_1 \iota_f \phi -(-1)^{l(m+k)}(\pi \circ_2 \iota_f \phi)\circ_1 \psi +(-1)^{mk} \big( (\pi \circ_2 \iota_f \psi) \circ_1 \phi - (-1)^{k(m+l)}(\pi \circ_2 \phi) \circ_1 \iota_f \psi \big) \nonumber\\
           % &= \sum_{i=1}^k (-1)^{(i-1)m}(\pi \circ_2 \psi) \circ_1 (\phi \circ_i f) - (-1)^{l(m+k)}\sum_{i=1}^k (-1)^{(i-1)m} (\pi \circ_2 (\phi \circ_i f))\circ_1 \psi \nonumber \\
           % & \quad + (-1)^{mk}  \sum_{i=1}^l (-1)^{(i-1)m}(\pi \circ_2 (\psi \circ_i f)) \circ_1 \phi - (-1)^{kl} \sum_{i=1}^l (-1)^{(i-1)m} (\pi \circ_2 \phi) \circ_1 (\psi \circ_i f) .
        %\end{align}
        %Similarly, we expand the LHS and then rearrange the sums to obtain the following:
        \begin{align*} %\label{lhs1}
            &\rho (f)[\phi, \psi]_\pi \\
          %  = \iota_f [\phi , \psi]_\pi \\
            &= \iota_f \big((\pi \circ_2 \psi)\circ_1 \phi - (-1)^{kl}~\!(\pi \circ_2 \phi)\circ_1 \psi\big) \\
            &= \sum_{i=1}^{k+l} (-1)^{(i-1)m}((\pi \circ_2 \psi)\circ_1 \phi) \circ_i f - (-1)^{kl}~\!\sum_{i=1}^{k+l}(-1)^{(i-1)m}((\pi \circ_2 \phi)\circ_1 \psi)\circ_i f \\
            &= \sum_{i=1}^{k} (-1)^{(i-1)m}((\pi \circ_2 \psi)\circ_1 \phi)\circ_i f +\sum_{i=k+1}^{k+l} (-1)^{(i-1)m}((\pi \circ_2 \psi)\circ_1 \phi)\circ_i f  \\
            & \quad - (-1)^{kl} ~\! \sum_{i=1}^{l}(-1)^{(i-1)m}((\pi \circ_2 \phi)\circ_1 \psi)\circ_i f - (-1)^{kl}~\! \sum_{i=l+1}^{l+k}(-1)^{(i-1)m}((\pi \circ_2 \phi)\circ_1 \psi)\circ_i f \\
            &= \sum_{i=1}^{k} (-1)^{(i-1)m}((\pi \circ_2 \psi)\circ_1 \phi)\circ_i f + (-1)^{km}~\! \sum_{i=1}^l (-1)^{(i-1)m}((\pi \circ_2 \psi)\circ_1\phi)\circ_{k+i} f  \\
            & \quad - (-1)^{kl}~\! \sum_{i=1}^l (-1)^{(i-1)m}((\pi \circ_2 \phi)\circ_1 \psi) \circ_i f - (-1)^{l(k+m)}~\! \sum_{i=1}^k (-1)^{(i-1)m}((\pi \circ_2 \phi)\circ_1 \psi)\circ_{l+i} f  \\
            &= \sum_{i=1}^k (-1)^{(i-1)m}(\pi \circ_2 \psi) \circ_1 (\phi \circ_i f)  + (-1)^{km} ~\! \sum_{i=1}^l (-1)^{(i-1)m}(\pi \circ_2 (\psi \circ_i f)) \circ_1 \phi  \\ 
            & \quad  - (-1)^{kl}~\! \sum_{i=1}^l (-1)^{(i-1)m} (\pi \circ_2 \phi) \circ_1 (\psi \circ_i f) - (-1)^{l(k+m)}~\!\sum_{i=1}^k (-1)^{(i-1)m}~\! (\pi \circ_2 (\phi \circ_i f))\circ_1 \psi\\
            &=  (\pi \circ_2 \psi) \circ_1 \iota_f \phi  +(-1)^{mk}~\! \big( (\pi \circ_2 \iota_f \psi) \circ_1 \phi - (-1)^{k(m+l)}~\! (\pi \circ_2 \phi) \circ_1 \iota_f \psi \big) -(-1)^{l(k+m)}~\!(\pi \circ_2 \iota_f \phi)\circ_1 \psi\\
            &= [\iota_f \phi , \psi]_\pi + (-1)^{mk}~\! [\phi, \iota_f \psi]_\pi = [\rho(f) \phi , \psi]_\pi + (-1)^{mk}~\! [\phi , \rho(f) \psi]_\pi.
            %&= \Big( \sum_{i=1}^{k} (-1)^{(i-1)m}((\pi \circ_2 \psi)\circ_1 \phi)\circ_i f - (-1)^{l(k+m)}\sum_{i=1}^k (-1)^{(i-1)m}((\pi \circ_2 \phi)\circ_1 \psi)\circ_{l+i} f \Big)\\
            %& \quad + (-1)^{mk} \Big( \sum_{i=1}^l (-1)^{(i-1)m}((\pi \circ_2 \psi)\circ_1 \phi)\circ_{k+i} f - (-1)^{k(m+l)} \sum_{i=1}^k (-1)^{(i-1)m}((\pi \circ_2 \phi)\circ_1 \psi)\circ_{l+i} f \Big).
        \end{align*}
        Hence the proof.
       % Using the properties of $\circ_i$, we have $(\pi \circ_2 \psi) \circ_1 (\phi \circ_i f) = ((\pi \circ_2 \psi) \circ_1 \phi) \circ_i f$ and $(\pi \circ_2(\phi \circ_i f)) \circ_1 \psi= ((\pi \circ_2 \phi) \circ_1 \psi) \circ_{l+i} f $, for $1 \leq i \leq k$. This implies that the first two terms of (\ref{rhs1}) are equal to the first and the fourth terms of (\ref{lhs1}) respectively. Similarly for $1 \leq i \leq l$, we have $(\pi \circ_2 (\psi \circ_i f))\circ_1\phi = ((\pi \circ_2 \psi)\circ_1 \phi)\circ_{k+i} f$ and $(\pi \circ_2 \phi) \circ_1 (\psi \circ_i f)= ((\pi \circ_2 \phi) \circ_1 \psi) \circ_i f $. Thus, the last two terms of (\ref{rhs1}) are equal to the second and the fourth terms of (\ref{lhs1}) respectively. Hence $\rho(f) [\phi , \psi]_{\pi} = [ \rho (f) \phi , \psi]_\pi + (-1)^{mk} [\phi , \rho (f) \psi]_\pi$ proving (\ref{action_2}). 
    \end{proof}

    Recall that an action of a graded Lie algebra $(\mathcal{A}, [~,~]_{\mathcal{A}})$ on another graded Lie algebra $( \mathcal{B}, [~,~]_{\mathcal{B}})$ is given by a degree zero bilinear map $\rho : \mathcal{A} \times \mathcal{B} \rightarrow \mathcal{B}$, $(a,b) \mapsto \rho(a) b$ that satisfies
        \begin{align}
            \rho ([a, a']_\mathcal{A})(b) =~& \rho (a) \rho (a')(b) - (-1)^{|a||a'|} ~ \! \rho (a') \rho (a)(b), \label{action1}\\
            \rho (a) ([b, b']_\mathcal{B}) =~& [ \rho (a) b, b']_\mathcal{B} + (-1)^{|a||b|} ~ \! [b, \rho (a) b']_\mathcal{B}, \label{action2}
    \end{align}
    for all homogeneous elements $a,a' \in \mathcal{A}$ and $b, b' \in \mathcal{B}$. Thus, it follows from Theorem \ref{rho-theorem} that the map $\rho$ given by (\ref{rho-map})  defines an action of the Gerstenhaber-Voronov algebra $(\mathcal{P}_{\bullet +1}, [~,~]_{\mathsf{GV}})$ on the cup Lie algebra $(\mathcal{P}_\bullet , [~,~]_\pi)$. As a consequence, the graded vector space $\mathcal{P}_{\bullet +1} \oplus \mathcal{P}_\bullet$ carries a graded Lie algebra structure with the bracket 
        \begin{align} \label{semidirect-product}
            [(f, \phi), (g, \psi )]_\ltimes := \big( [f,g]_\mathsf{GV} ~ \!, ~ \! [\phi , \psi]_\pi + \iota_f \psi - (-1)^{m n}~ \! \iota_g \phi \big),
        \end{align}
    for $(f,\phi)\in \mathcal{P}_{m+1} \oplus \mathcal{P}_m$ and $(g, \psi) \in \mathcal{P}_{n+1} \oplus \mathcal{P}_n$. Additionally, the Gerstenhaber-Voronov algebra $(\mathcal{P}_{\bullet +1}, [~,~]_{\mathsf{GV}})$ is a graded Lie subalgebra and $\mathcal{P}_\bullet$ is an ideal of the graded Lie algebra $(\mathcal{P}_{\bullet +1}\oplus \mathcal{P}_\bullet , [~,~]_\ltimes)$.

    We are now in a position to define the Fr\"{o}licher-Nijenhuis bracket associated with a nonsymmetric operad endowed with multiplication. Let $\mathcal{P}$ be a nonsymmetric operad and $\pi$ be a multiplication. For each $m,n\geq 1$, we define a bracket $[~,~]_\mathsf{FN}:\mathcal{P}_m\times \mathcal{P}_n \rightarrow \mathcal{P}_{m+n}$ by
        \begin{align} \label{FN-bracket}
            [f,g]_\mathsf{FN}:=[f,g]_\pi + (-1)^m ~ \! \iota_{(\delta_\pi f)} g - (-1)^{(m+1)n} ~ \! \iota_{(\delta_\pi g)} f ,
        \end{align}
    for $f\in \mathcal{P}_m$ and  $g\in \mathcal{P}_n$. The bracket $[~,~]_\mathsf{FN}$ is called the {\bf Fr\"{o}licher-Nijenhuis bracket} associated to the nonsymmetric operad $\mathcal{P}$ with multiplication $\pi$. Using Proposition \ref{pi-bracket}, the Fr\"{o}licher-Nijenhuis bracket $[~,~]_\mathsf{FN}$ can be expressed as 
        \begin{align*}
            [f,g]_\mathsf{FN} &= \iota_f (\delta_\pi g)+ (-1)^{(m-1)} ~ \! \cancel{\iota_{\delta_\pi f}g }+ (-1)^m ~ \! \delta_\pi (\iota_f g) + (-1)^m ~ \! \cancel{\iota_{\delta_\pi f}g} - (-1)^{(m+1)n} ~ \! \iota_{\delta_\pi g}f  \\
            &= [f, \delta_\pi g]_\mathsf{GV} + (-1)^m ~ \! \delta_\pi (\iota_f g).
        \end{align*}
    Moreover, using the explicit form of the cup bracket $[~,~]_\pi$ and the contraction operator $\iota$ in the defining identity (\ref{FN-bracket}), we have
        \begin{align}\label{fn-partial}
            [f,g]_\mathsf{FN}=~& (\pi \circ_2 g) \circ_1 f - (-1)^{mn} ~ \! (\pi \circ_2 f) \circ_1 g + (-1)^m ~ \! \sum_{i=1}^n (-1)^{(i-1)m} ~ \! g \circ_i (\delta_\pi f)\\
            & \qquad - (-1)^{(m+1)n} ~ \! \sum_{i=1}^m (-1)^{(i-1)n} ~ \! f \circ_i (\delta_\pi g). \nonumber
        \end{align}

        \medskip

        For a nonsymmetric operad $\mathcal{P}$ with a multiplication $\pi$, the following result gives a connection between the Fr\"{o}licher-Nijenhuis bracket $[~,~]_\mathsf{FN}$ and the Gerstenhaber-Voronov bracket $[~,~]_\mathsf{GV}$.

    \begin{proposition}\label{prop-fn-gv}
        For any $f \in \mathcal{P}_m$ and $g \in \mathcal{P}_n$, we have
            \begin{align} \label{FN-GV}
                \delta_\pi [f,g]_\mathsf{FN} = [\delta_\pi f , \delta_\pi g ]_{\mathsf{GV}}.
            \end{align}
    \end{proposition}
    \begin{proof}
        By a direct calculation, we get
        \begin{align*}
            \delta_\pi [f,g]_\mathsf{FN} &= \delta_\pi \big( [f,g]_\pi + (-1)^m ~ \! \iota_{\delta_\pi f} g - (-1)^{(m+1 )n} ~ \! \iota_{\delta_\pi g}f \big)\\
            &= \delta_\pi (\iota_f (\delta_\pi g)) + (-1)^{m-1} ~ \! \cancel {\delta_\pi (\iota_{\delta_\pi f} g)} + (-1)^m ~ \! \cancel{\delta_\pi(\iota_{\delta_\pi f} g)} - (-1)^{(m+1 )n} ~ \! \delta_\pi ( \iota_{\delta_\pi g}f )\\
            &= \delta_\pi \big( \iota_f (\delta_\pi g) - (-1)^{(m+1)n} ~ \! \iota_{\delta_\pi g }f  \big)
            = \delta_\pi [f, \delta_\pi g]_{\mathsf{GV}}.
        \end{align*}
        Hence the result follows as $\delta_\pi = - [\pi, - ]_\mathsf{GV}$ is a graded derivation for the Gerstenhaber-Voronov bracket and $\delta_\pi^2 = 0.$
    \end{proof}

    \begin{theorem} \label{FN-gla}
        Let $\mathcal{P}$ be a nonsymmetric operad and $\pi$ be a multiplication on $\mathcal{P}$.
        \begin{enumerate}
            \item Then the Fr\"{o}licher-Nijenhuis bracket makes $(\mathcal{P}_\bullet, [~,~]_\mathsf{FN})$ into a graded Lie algebra, called the {\bf Fr\"{o}licher-Nijenhuis algebra}.
            \item Moreover, the differential map $\delta_\pi : \mathcal{P}_\bullet \rightarrow \mathcal{P}_{\bullet + 1}$ is a graded Lie algebra homomorphism from the Fr\"{o}licher-Nijenhuis algebra $(\mathcal{P}_\bullet , [~,~]_\mathsf{FN} )$ to the Gerstenhaber-Voronov algebra $(\mathcal{P}_{\bullet +1}, [~,~]_\mathsf{GV})$.
        \end{enumerate}
    \end{theorem}
    \begin{proof}
        (1) From the expression (\ref{FN-bracket}), it is easy to see that $[~,~]_\mathsf{FN}$ is graded skew-symmetric. To prove the graded Jacobi identity of $[~,~]_\mathsf{FN}$, we consider the injective map  $\Theta : \mathcal{P}_\bullet \rightarrow \mathcal{P}_{\bullet +1} \oplus \mathcal{P}_\bullet $ of graded vector spaces by $\Theta(f)= ((-1)^m ~ \! \delta_\pi f ,  f)$, for $f \in \mathcal{P}_m$. It is important to note that the space $\mathcal{P}_{\bullet +1} \oplus \mathcal{P}_\bullet$ has a graded Lie algebra structure with the bracket $[~,~]_\ltimes$ given in (\ref{semidirect-product}).
        Then for any $f\in \mathcal{P}_m,~ \! g \in \mathcal{P}_n$ and $h \in \mathcal{P}_k$, we observe that
        \begin{align*}
            [\Theta(f), \Theta(g)]_\ltimes =~& [((-1)^m ~ \! \delta_\pi f, f), ((-1)^n ~ \! \delta_\pi g, g)]_\ltimes\\
            =~& \big( (-1)^{m+n} ~ \! [\delta_\pi f, \delta_\pi g]_\mathsf{GV} ~ \! , ~ \! [f,g]_\pi + (-1)^m ~ \! \iota_{\delta_\pi f } g - (-1)^{(m+1)n}~ \!  \iota_{\delta_\pi g} f  \big)\\
            \stackrel{(\ref{FN-GV})}{=}& ~\!\big( (-1)^{m+n} ~ \!  \delta_\pi [f,g]_\mathsf{FN} ~ \! , ~ \! [f,g]_\mathsf{FN} \big)  
            = \Theta [f,g]_\mathsf{FN}
        \end{align*}
and hence, $[ [\Theta(f), \Theta(g)]_\ltimes, \Theta (h) ]_\ltimes = \Theta ([[f, g]_\mathsf{FN}, h]_\mathsf{FN})$.
    %    \begin{align*}
    %        [\Theta(f), [\Theta(g), \Theta(h)]_\ltimes]_\ltimes = \Theta \big( [f, [g,h]_\mathsf{FN}]_\mathsf{FN} \big) .
     %   \end{align*}
        Since $\Theta$ is injective and the bracket $[~,~]_\ltimes$ satisfies the graded Jacobi identity,
     the bracket $[~,~]_\mathsf{FN}$ also satisfies the same. Hence $(\mathcal{P}_\bullet, [~,~]_\mathsf{FN})$ is a graded Lie algebra.

        (2) Follows from Proposition \ref{prop-fn-gv}.
    \end{proof}

    Let $N \in \mathcal{P}_1$ be an arbitrary element. Then it follows from the expression (\ref{fn-partial}) that
        \begin{align*}
            [N,N]_\mathsf{FN}= 2 ~ \! \big( (\pi \circ_2 N)\circ_1 N - N \circ_1 (\pi \circ_1 N + \pi \circ_2 N - N \circ_1 \pi)\big).
        \end{align*}
    This shows that $N$ is a Nijenhuis element for the multiplication $\pi$ if and only if $[N, N]_\mathsf{FN}=0$. Hence, Nijenhuis elements for the multiplication $\pi$ can be characterized by  Maurer-Cartan elements in the Fr\"{o}licher-Nijenhuis algebra $(\mathcal{P}_\bullet, [~,~]_\mathsf{FN})$. As a result, one may define a cohomology theory associated with a Nijenhuis element. More precisely, let $N$ be a Nijenhuis element for the multiplication $\pi$ in a given nonsymmetric operad $\mathcal{P}$. Then for each $n \geq 1$, there is a map
    \begin{align*}
        d_N : \mathcal{P}_n \rightarrow \mathcal{P}_{n+1} ~~~ \text{ given by } ~~~ d_N (f) = [N, f]_\mathsf{FN}, \text{ for } f \in \mathcal{P}_n.
    \end{align*}
    Since $N$ is a Nijenhuis element for the multiplication $\pi$ (i.e., $[N, N ]_\mathsf{FN} = 0$), it turns out that $(d_N)^2 = 0$. Hence $\{ \mathcal{P}_\bullet , d_N \}$ is a cochain complex. The corresponding cohomology groups are denoted by $H^\bullet_{\mathcal{P}, \pi} (N)$, and they are said to be the cohomology groups associated with the Nijenhuis element $N$.

    \begin{proposition}
        Let $N \in \mathcal{P}_1$ be a Nijenhuis element for the multiplication $\pi$. Define $\pi_N \in \mathcal{P}_2$ by 
            \begin{align*}
                \pi_N := \pi \circ_1 N + \pi \circ_2 N - N \circ_1 \pi .
            \end{align*}
            Then $\pi_N$ is also a multiplication on the nonsymmetric operad $\mathcal{P}$. Moreover, $N$ is a map from the multiplication $\pi_N$ to the multiplication $\pi$.
    \end{proposition}
    \begin{proof}
        First observe that ~ $\delta_\pi N= -[\pi, N]_\mathsf{GV} = - (\iota_\pi N - \iota_N \pi) = \pi \circ_1 N + \pi \circ_2 N - N \circ_1 \pi = \pi_N$. Hence
            \begin{align*}
                 [\pi_N, \pi_N]_{\mathsf{GV}}=~ [\delta_\pi N \! , ~ \! \delta_\pi N]_{\mathsf{GV}} 
                 \stackrel{(\ref{FN-GV})}{=}~\delta_\pi [N,N]_\mathsf{FN}
                 =~ 0
           \end{align*}
        which shows that $\pi_N$ is a multiplication on $\mathcal{P}$. The last part follows from the defining identity of the Nijenhuis element $N$, namely, $N \circ_1 \pi_N = (\pi \circ_2 N) \circ_1 N$.
    \end{proof}

Let $N$ be a Nijenhuis element for the multiplication $\pi$. In the following, we relate the cohomology associated with the Nijenhuis element $N$ and the cohomology associated with the induced multiplication $\pi_N$. For this, for each $n \geq 1$, we first define a map $\Psi_n : \mathcal{P}_n \rightarrow \mathcal{P}_{n+1}$ by $\Psi_n (f) = (-1)^{n+1} ~ \! \delta_\pi (f)$, for $f \in \mathcal{P}_n$. Then we have
\begin{align*}
    (\delta_{\pi_N} \circ \Psi_n)(f)  =~& - (-1)^{n+1} ~ \! [\pi_N, \delta_\pi f]_\mathsf{GV} \\
    =~& (-1)^{n} ~ \! [\delta_\pi N, \delta_\pi f]_\mathsf{GV} \\
    \stackrel{(\ref{FN-GV})}{=}& ~\!(-1)^n ~ \! \delta_\pi [N, f]_\mathsf{FN} = (-1)^n ~ \! \delta_\pi (d_N f) = (\Psi_{n+1} \circ d_N) (f).
\end{align*}
Here $\delta_{\pi_N}$ is the coboundary operator in the cochain complex $\{ \mathcal{P}_\bullet , \delta_{\pi_N} \}$ associated with the multiplication $\pi_N$. Hence, we get the following.

\begin{proposition}
    Let $N$ be a Nijenhuis element for the multiplication $\pi$ in a nonsymmetric operad $\mathcal{P}$. Then there is a homomorphism $H^\bullet_{\mathcal{P}, \pi} (N) \rightarrow H^{\bullet +1}_{\mathcal{P}} (\pi_N)$ from the cohomology associated with the Nijenhuis element $N$ to the cohomology associated with the induced multiplication $\pi_N$.
\end{proposition}

\medskip

\section{Rota-Baxter elements and the derived bracket}
Given a nonsymmetric operad with a multiplication $\pi$, here we first consider Rota-Baxter elements of weight $\lambda \in {\bf k}$ for the multiplication $\pi$ as a generalization of Rota-Baxter operators on associative algebras \cite{das-rota,guo-book}. We also construct the graded Lie algebra (resp. differential graded Lie algebra) whose Maurer-Cartan elements are precisely Rota-Baxter elements of weight $0$ (resp. of weight $\lambda \neq 0$). Using this Maurer-Cartan characterization, we define the cohomology of a given Rota-Baxter element of weight $\lambda$. On the other hand, we show that a Rota-Baxter element of weight $\lambda$ induces a new multiplication. Finally, the cohomology associated with this induced multiplication with coefficients in a suitable representation is isomorphic to the cohomology of the given Rota-Baxter element of weight $\lambda$. We end this section by considering averaging elements for the multiplication $\pi$ and finding the Maurer-Cartan characterization of averaging elements.

\medskip

    Let $\mathcal{P}$ be a nonsymmetric operad and $\pi$ be a multiplication on $\mathcal{P}$. An element $R \in \mathcal{P}_1$ is said to be a {\bf Rota-Baxter element of weight $\lambda \in \mathbf{k}$} for the multiplication $\pi$ if
        \begin{align*}
            (\pi \circ_2 R) \circ_1 R = R \circ_1 (\pi \circ_1 R + \pi \circ_2 R + \lambda ~ \! \pi).
        \end{align*}

    \begin{remark}
        Let $\mathcal{P}= \mathrm{End}_A$ be the endomorphism operad and $\pi \in \mathcal{P}_2 = (\mathrm{End}_A)_2$ be a multiplication representing the associative algebra structure $(A, ~\! \cdot ~\!)$ on the vector space $A$. Then an element $R \in (\mathrm{End}_A)_1 = \mathrm{Hom}(A, A)$ is a Rota-Baxter element of weight $\lambda \in {\bf k}$ for the multiplication $\pi$ if and only if $R$ is a Rota-Baxter operator of weight $\lambda \in {\bf k}$ on the associative algebra $(A, ~\! \cdot ~\!) $ \cite{das-rota, guo-book}, i.e.,
        \begin{align*}
            R (a) \cdot R (b) = R (R(a) \cdot b + a \cdot R(b) + \lambda ~\! a \cdot b), \text{ for all } a, b \in A.
        \end{align*}
        Hence, a Rota-Baxter element generalizes Rota-Baxter operators on associative algebras.
        %associated with the vector space $A$. Then an element $R \in (\mathrm{End}_A)_1= \mathrm{Hom}(A, A)$ is a Rota-Baxter element of weight $\lambda$ for a multiplication $\pi$ if and only if the linear map $R: A \rightarrow A$ is a Rota-Baxter operator of weight $\lambda$ on the associative algebra $(A, \pi)$.
    \end{remark}

    \begin{proposition}
        Let $\mathcal{P} = (\mathcal{P}, \circ, \mathbbm{1})$ be a nonsymmetric operad and $\pi$ be a multiplication on $\mathcal{P}$. Let $R \in \mathcal{P}_1$ be a Rota-Baxter element of weight $\lambda$ for the multiplication $\pi$. Then $-\lambda \mathbbm{1}-R \in \mathcal{P}_1$ is also a Rota-Baxter element of same weight for the multiplication $\pi$.
    \end{proposition}
    \begin{proof}
We observe that
            \begin{align*}
                 &\big( \pi \circ_2  (-\lambda \mathbbm{1}-R)  \big) \circ_1 (-\lambda \mathbbm{1}-R) - (-\lambda \mathbbm{1}-R) \circ_1 \big( \pi \circ_1 (-\lambda \mathbbm{1}-R) + \pi \circ_2 (-\lambda \mathbbm{1}-R) + \lambda \pi   \big) \\
                &~=   \big( \pi \circ_2 (\lambda \mathbbm{1} +R)  \big) \circ_1 (\lambda \mathbbm{1}+R) - (\lambda \mathbbm{1}+R)\circ_1 \big( \pi \circ_1 (\lambda \mathbbm{1}+R)+ \pi \circ_2 (\lambda \mathbbm{1}+R)-\lambda \pi  \big)\\
                &~=  (\lambda \pi + \pi \circ_2 R)\circ_1 (\lambda \mathbbm{1} +R)- (\lambda \mathbbm{1}+R)- (\lambda \mathbbm{1}+ R)\circ_1 (\lambda \pi + \pi \circ_1 R + \lambda \pi + \pi \circ_2 R - \lambda \pi)\\
                &~= \lambda^2 \pi + \lambda (\pi \circ_1 R) + \lambda (\pi \circ_2 R)+ (\pi \circ_2 R)\circ_1 R - \lambda^2 \pi - \lambda (\pi \circ_1 R) - \lambda(\pi \circ_2 R)-\lambda (R \circ_1 \pi) \\
                  &~ \quad  - R \circ_1 (\pi \circ_1 R)- R \circ_1 (\pi \circ_2 R )\\ 
                &~=  (\pi \circ_2 R)\circ_1 R- R \circ_1 (\pi \circ_1 R +\pi \circ_2 R +\lambda \pi)
                =  0
            \end{align*}
        which shows that $-\lambda \mathbbm{1}-R $ is also a Rota-Baxter element of weight $\lambda$ for the multiplication $\pi$.
        \end{proof}

    Our aim in this section is to construct a graded Lie algebra (resp. a differential graded Lie algebra) whose Maurer-Cartan elements are precisely Rota-Baxter elements of weight $0$ (resp. of a nonzero weight $\lambda \in {\bf k}$) for a given multiplication $\pi$. Let $\mathcal{P}$ be a nonsymmetric operad and $\pi$ be a multiplication on $\mathcal{P}$. For each $m,n \geq 1$, we define a bracket $[~,~]_\mathsf{D} : \mathcal{P}_m \times \mathcal{P}_n \rightarrow \mathcal{P}_{m+n}$ by
        \begin{align} \label{derived-bracket}
            [f,g]_\mathsf{D} := [f,g]_\pi + \iota_{\theta f} g - (-1)^{mn} ~ \! \iota_{\theta g} f, 
        \end{align}
    for $f \in \mathcal{P}_m$ and $g \in\mathcal{P}_n$. Here $\theta : \mathcal{P}_n \rightarrow \mathcal{P}_{n+1}$ (for $n \geq 1$) is the map given in (\ref{theta}). The bracket $[~,~]_\mathsf{D}$ is called the {\bf derived bracket} associated to the nonsymmetric operad $\mathcal{P}$ with multiplication $\pi$. 

%    Then we have the following results.
    \begin{proposition} \label{D-GV}
        For a nonsymmetric operad $\mathcal{P} $ with multiplication $\pi$, we have
        \begin{align} \label{theta-D}
            \theta[f,g]_\mathsf{D} = [\theta f, \theta g]_{\mathsf{GV}}, \text{ for }  f \in \mathcal{P}_m, ~ \! g \in \mathcal{P}_n.
        \end{align}
    \end{proposition}
    \begin{proof} We have
        \begin{align*}
            \theta [f,g]_\mathsf{D} =~& \theta \big( [f,g]_\pi + \iota_{\theta f} g - (-1)^{mn} ~ \! \iota_{\theta g} f  \big) \\
            \stackrel{(\ref{theta-pi})}{=}& ~\! (-1)^n~\! [\theta f , g]_\pi + [f, \theta g]_\pi + \theta (\iota_{\theta f} g) - (-1)^{mn} ~ \! \theta (\iota_{\theta g} f)\\
            \stackrel{~\!(\ref{first})~\!}{=}&  \iota_{\theta f} \theta g - \cancel{\theta (\iota_{\theta f} g)} - (-1)^{m(n+1)} ~ \! (-1)^m ~ \!  \iota_{\theta g} \theta f + (-1)^{m(n+1)} ~ \! (-1)^m ~ \!  \cancel{\theta (\iota_{\theta g} f)} + \cancel{\theta(\iota_{\theta f} g )}- (-1)^{mn} ~ \!  \cancel{\theta (\iota_{\theta g} f)}\\
            =~& \iota_{\theta f} \theta g - (-1)^{mn} ~ \! \iota_{\theta g} \theta f
            = [\theta f, \theta g]_\mathsf{GV}
        \end{align*}   
        which proves the result.
    \end{proof}

    \begin{theorem}
        Let $\mathcal{P}$ be a nonsymmetric operad and $\pi$ be a multiplication on $\mathcal{P}$.
            \begin{enumerate}
                \item Then the derived bracket makes $(\mathcal{P}_\bullet, [~,~]_\mathsf{D})$ into a graded Lie algebra, called the {\bf derived algebra}.
                \item Moreover, the map $\theta : \mathcal{P}_\bullet \rightarrow \mathcal{P}_{\bullet +1}$ is a graded Lie algebra homomorphism from the derived algebra $(\mathcal{P}_\bullet, [~,~]_\mathsf{D})$ to the Gerstenhaber-Voronov algebra $(\mathcal{P}_{\bullet +1}, [~,~]_\mathsf{GV})$.
            \end{enumerate}
    \end{theorem}
    \begin{proof}
        (1) It is easy to see from the expression (\ref{derived-bracket}) that the derived bracket $[~,~]_\mathsf{D} $ is graded skew-symmetric. To prove the graded Jacobi identity, we consider the injective map $\Phi: \mathcal{P}_\bullet \rightarrow \mathcal{P}_{\bullet +1} \oplus \mathcal{P}_\bullet$ of graded vector spaces by $\Phi (f) = (\theta (f), f)$, for $f \in \mathcal{P}_m$. Then for $f \in \mathcal{P}_m, ~ \! g \in \mathcal{P}_n$ and $h \in \mathcal{P}_k$, we have
        \begin{align*}
            [\Phi (f), \Phi (g)]_\ltimes =~& [(\theta f, f),(\theta g, g)]_\ltimes
            = \big( [\theta f, \theta g]_\mathsf{GV} ~ \! , ~ \! [f,g]_\pi + \iota_{\theta f} g - (-1)^{mn} ~ \! \iota_{\theta g} f \big)\\
            \stackrel{(\ref{theta-D})}{=}& ~\! \big( \theta [f,g]_\mathsf{D}, [f,g]_\mathsf{D} \big) 
            = \Phi [f,g]_\mathsf{D}
        \end{align*}
        and hence, $[ [ \Phi (f), \Phi (g)]_\ltimes, \Phi (h) ]_\ltimes = \Phi ([[f, g]_\mathsf{D}, h]_\mathsf{D})$. Since $\Phi$ is injective and $[~,~]_\ltimes$ satisfies the graded Jacobi identity, the bracket $[~,~]_\mathsf{D} $ also satisfies the same. Hence $(\mathcal{P}_\bullet , [~,~]_\mathsf{D})$ is a graded Lie algebra.
        
        (2) Follows from Proposition \ref{D-GV}.   
    \end{proof}

    Let $R \in \mathcal{P}_1$ be an arbitrary element. Then it follows from (\ref{derived-bracket}) that
        \begin{align*}
            [R,R]_\mathsf{D}= 2 ~ \! \big( (\pi \circ_2 R) \circ_1 R - R \circ_1 (\pi \circ_1 R + \pi \circ_2 R)  \big).
        \end{align*}
    This shows that $R$ is a Rota-Baxter element of weight zero if and only if $[R, R]_\mathsf{D}=0$. In other words, Maurer-Cartan elements in the derived algebra $(\mathcal{P}_\bullet, [~,~]_\mathsf{D})$ are precisely Rota-Baxter elements of weight $0$ for the multiplication $\pi$.

    Let $\lambda \in \mathbf{k}$ be a fixed scalar. For any $n \geq 1$, we now define a map $d_\lambda : \mathcal{P}_n \rightarrow \mathcal{P}_{n+1}$ by
        \begin{align} \label{d_lambda}
            d_\lambda (f)= -\lambda ~ \!  \iota_\pi f ,  \text{ for } f \in \mathcal{P}_n.
        \end{align}
    Thus, in terms of the map $D$ given in (\ref{derivation-map}), we have $d_\lambda = \lambda ~ \! D$.

    \begin{proposition}
        For any $\lambda \in {\bf k}$, the map $d_\lambda$ is a differential (i.e., $d_\lambda ^2 =0$). Moreover, $d_\lambda$ is a graded derivation for the derived bracket $[~,~]_\mathsf{D}$.
    \end{proposition}
    \begin{proof}
         Since $D^2 =0$, we have $d_\lambda ^2 =0$. For $\lambda = 0$, it follows that $d_\lambda = 0$ and hence it is a graded derivation for the derived bracket. For nonzero $\lambda \in {\bf k}$, since $d_\lambda = \lambda ~ \! D$, it is enough to prove that $D$ is a graded derivation for $[~,~]_\mathsf{D}$. For any $f \in \mathcal{P}_m$, since
            \begin{align*}
                0 = D^2(f)= D(\delta_\pi f -(-1)^m ~ \! \theta f) = - \theta^2 (f) + (-1)^m ~ \! (\theta \delta_\pi - \delta_\pi \theta)(f),
            \end{align*}
        we get that
            \begin{align} \label{theta-square}
                \theta^2 (f) = (-1)^m ~ \! (\theta \delta_\pi- \delta_\pi \theta) (f).
            \end{align}
        Hence for any $f \in \mathcal{P}_m$ and $ g \in \mathcal{P}_n$,
            \begin{align*}
                &[D(f), g]_\mathsf{D} + (-1)^m ~ \! [f, D(g)]_\mathsf{D}  \\
                &=   [\delta_\pi f ,g]_\mathsf{D} + (-1)^{m-1} ~ \! [\theta f, g]_\mathsf{D} + (-1)^m ~ \! [f, \delta_\pi  g]_\mathsf{D} + (-1)^{m+n-1} ~ \! [f , \theta g]_\mathsf{D} \quad (\text{by } (\ref{delta-D-theta})) \\
                &= \Big( [\delta_\pi f, g]_\pi + \iota_{\theta \delta_\pi f} g - (-1)^{(m+1)n} ~ \! \cancel{\iota_{\theta g} \delta_\pi f}  \Big)  \\
                & \quad + (-1)^{m-1} \Big( \cancel{\iota_{\theta f} \delta_\pi g} +(-1)^m ~ \! \iota_{\delta_\pi \theta f}g +(-1)^{m+1} ~ \! \delta_\pi (\iota_{\theta f} g) + \iota_{\theta^2 f} g - (-1)^{(m+1)n} ~ \! \iota_{\theta g} \theta f  \Big)  \\
                & \quad + (-1)^m ~ \! \Big( [f, \delta_\pi g]_\pi + \cancel{\iota_{\theta f} \delta_\pi g} - (-1)^{m(n+1)} ~ \! \iota_{\theta \delta_\pi g}f \Big)  \\
                & \quad + (-1)^{(m+1)n} ~ \! \Big( \cancel{\iota_{\theta g} \delta_\pi f} + (-1)^n ~ \! \iota_{\delta_\pi \theta g} f + (-1)^{n+1} ~ \! \delta_\pi(\iota_{\theta g}f) + \iota_{\theta^2 g} f - (-1)^{m(n+1)} ~ \! \iota_{\theta f} \theta g  \Big) \\
                &= [\delta_\pi f ,g]_\pi + (-1)^m ~ \! [f, \delta_\pi g]_\pi + \delta_\pi \iota_{\theta f} g - (-1)^{mn} ~ \! \delta_\pi \iota_{\theta g} f + (-1)^{m+n-1} ~ \! \big( \iota_{\theta f} \theta g - (-1)^{mn} ~ \! \iota_{\theta g} \theta f \big) \\
                & \quad + \big( \iota_{\theta \delta_\pi f} g - \iota_{\delta_\pi \theta f} g - (-1)^m ~ \! \iota_{\theta^2 f} g  \big) - (-1)^{mn} ~ \! \big( \iota_{\theta \delta_\pi g} f - \iota_{\delta_pi \theta g} f - (-1)^n ~ \! \iota_{\theta^2 g} f \big) \\
                &\stackrel{(\ref{theta-square})}{=} \delta_\pi\big( [f,g]_\pi + \iota_{\theta f}g -(-1)^{mn} ~ \! \iota_{\theta g} f \big) + (-1)^{m+n-1} ~ \! [\theta f, \theta g]_\mathsf{GV}\\
                &\stackrel{(\ref{theta-D})}{=} \delta_\pi [f,g]_\mathsf{D} +(-1)^{m+n-1} ~ \! \theta ([f,g]_\mathsf{D} )
                = D [f,g]_\mathsf{D}
            \end{align*}
        which shows that $D$ is a graded derivation for the bracket $[~,~]_\mathsf{D} $. This concludes the proof.
    \end{proof}
    The above proposition shows that the triple $(\mathcal{P}_\bullet , [~,~]_\mathsf{D}, d_\lambda)$ is a differential graded Lie algebra. Moreover, for any arbitrary element $R \in \mathcal{P}_1$, we have
        \begin{align*}
            d_\lambda (R) = \lambda ~ \! D(R) =- \lambda ~ \! \iota_\pi R = - \lambda ~ \!  R \circ_1 \pi
        \end{align*}
    and hence,
        \begin{align*}
            d_\lambda (R) + \frac{1}{2}[R,R]_\mathsf{D} = (\pi \circ_2 R) \circ_1 R - R \circ_1 (\pi \circ_1 R + \pi \circ_2 R + \lambda ~ \! \pi).
        \end{align*}
    This shows that $R$ is a Rota-Baxter element of weight $\lambda \in \mathbf{k}$ for the multiplication $\pi$ if and only if $R$ is a Maurer-Cartan element in the differential graded Lie algebra ($\mathcal{P}_\bullet, [~,~]_\mathsf{D}, d_\lambda$). Let $R$ be a fixed Rota-Baxter element of weight $\lambda$ for the multiplication $\pi$. Then for each $n \geq 1$, we define a map $d_R : \mathcal{P}_n \rightarrow \mathcal{P}_{n+1}$ by
    \begin{align}\label{map-dr}
        d_R (f) = - \big( d_\lambda (f) + [R, f]_\mathsf{D} \big), \text{ for } f \in \mathcal{P}_n.
    \end{align}
    Then it turns out that $(d_R)^2 = 0$. In other words, $\{ \mathcal{P}_\bullet, d_R \}$ is a cochain complex. The corresponding cohomology groups are denoted by $H^\bullet_{\mathcal{P}, \pi} (R)$ and they are said to be the cohomology groups associated with the Rota-Baxter element $R$ of weight $\lambda$.

    \medskip

    %In the following two results, we show that a Rota-Baxter element induces a new multiplication and a suitable representation of it. 
    The next result can be proved directly. However, we provide proof based on the Maurer-Cartan characterizations of multiplications and Rota-Baxter elements.
    
   %\begin{lemma}
    %    For any $R\in \mathcal{P}_1$, we have
    %    \begin{align} \label{theta-d}
     %       \theta (d_\lambda R)= -\lambda [\pi, \theta R]_\mathsf{GV}.
     %   \end{align}
    %\end{lemma}
    %\begin{proof}
   %     By direct calculation, we get
    %    \begin{align*}
     %       -\lambda [\pi, \theta R]_\mathsf{GV} =&~ -\lambda (\iota_\pi \theta R + \iota_{\theta R}\pi)\\
    %        =&~- \lambda \big( (\pi \circ_1 R)\circ_2 \pi + (\pi \circ_2 R) \circ_2 \pi - (\pi \circ_1 R)\circ_1 \pi - (\pi \circ_2 R) \circ_1 \pi \\
      %      & + \pi\circ_2 (\pi \circ_1 R) + \pi \circ_2 (\pi \circ_2 R)- \pi \circ_1 (\pi \circ_1 R)- \pi \circ_1 (\pi \circ_2 R)\big)\\
      %      = &~ \lambda \big( -(\pi \circ_2 R)\circ_2 \pi + (\pi \circ_1 R) \circ_1 \pi  \big) \qquad (\text{using the properties of $\circ_i$ and $\pi$})\\
      %      =&~ \lambda (\pi \circ_1 (R \circ_1 \pi)- \pi \circ_2(R \circ_1 \pi))\\
      %      =&~ \lambda \iota_{R \circ_1 \pi} \pi 
       %     = - \iota_{\lambda D(R)} \pi
       %     =  \theta (d_\lambda R)
       % \end{align*}  
      %  which proves the result.
  %  \end{proof}

    \begin{proposition}
        Let $R \in \mathcal{P}_1$ be a Rota-Baxter element of weight $\lambda$ for the multiplication $\pi$. We define an element $\pi_R \in \mathcal{P}_2$ by 
            \begin{align*}
                \pi_R:= \pi \circ_1 R + \pi \circ_2 R + \lambda ~ \! \pi .
            \end{align*}
        Then $\pi_R$ is also a multiplication on the nonsymmetric operad $\mathcal{P}$. Moreover, $R$ is a map from multiplication $\pi_R$ to the multiplication $\pi$.
    \end{proposition}
    \begin{proof}
        We have $\pi_R =   \pi \circ_1 R + \pi \circ_2 R + \lambda ~ \! \pi =  \iota_R \pi + \lambda ~ \! \pi =  - \theta R + \lambda ~ \! \pi.$ Also, observe that
        \begin{align}
            \lambda [\pi, \theta R]_\mathsf{GV} =&~ -\lambda (\iota_\pi \theta R + \iota_{\theta R}\pi) \nonumber\\
            =&~ \lambda \big( (\pi \circ_1 R)\circ_2 \pi + (\pi \circ_2 R) \circ_2 \pi - (\pi \circ_1 R)\circ_1 \pi - (\pi \circ_2 R) \circ_1 \pi  \nonumber\\
            & \qquad + \pi\circ_2 (\pi \circ_1 R) + \pi \circ_2 (\pi \circ_2 R)- \pi \circ_1 (\pi \circ_1 R)- \pi \circ_1 (\pi \circ_2 R)\big)  \nonumber\\
            = &~ - \lambda \big( -(\pi \circ_2 R)\circ_2 \pi + (\pi \circ_1 R) \circ_1 \pi  \big) \qquad   \nonumber\\ %(\text{using the properties of $\circ_i$ and $\pi$})
            =&~ - \lambda (\pi \circ_1 (R \circ_1 \pi)- \pi \circ_2(R \circ_1 \pi))  \nonumber\\
            =&~ - \lambda ~ \! \iota_{R \circ_1 \pi} \pi 
            =  \iota_{\lambda D(R)} \pi = - \theta (d_\lambda R). \label{theta-d}
        \end{align}
        Hence
        \begin{align*}
            [\pi_R , \pi_R]_\mathsf{GV} 
            = [\lambda ~ \! \pi - \theta R, \lambda ~ \! \pi - \theta R]_\mathsf{GV} 
            =~& \lambda^2 [\pi, \pi]_\mathsf{GV}- \lambda ~ \! [\pi,\theta R]_\mathsf{GV}- \lambda ~ \! [\theta R, \pi]_\mathsf{GV}+ [\theta R, \theta R]_\mathsf{GV}\\
            \stackrel{(\ref{theta-D})}{=}& -2 \lambda ~ \! [\pi, \theta R]_\mathsf{GV} + \theta [R,R]_\mathsf{D}\\
            \stackrel{(\ref{theta-d})}{=}& ~2 ~ \! \theta (d_\lambda R) + \theta [R,R]_\mathsf{D}
            = 2 ~ \! \theta (d_\lambda R+ \frac{1}{2}[R,R]_\mathsf{D})
            = 0
        \end{align*}
        which shows that $\pi_R$ is a multiplication on $\mathcal{P}$. Finally, the last part follows from the defining identity of the Rota-Baxter element of weight $\lambda$, namely, $R \circ_1 \pi_R = (\pi \circ_2 R) \circ_1 R$.
        %The last part follows from the identity $(\pi \circ_2 R) \circ_1 R = R \circ_1 \pi_R.$    
    \end{proof}

\begin{proposition}
    Let $\mathcal{P}$ be a nonsymmetric operad and $\pi$ be a multiplication on $\mathcal{P}$. Suppose $R$ is a Rota-Baxter element of weight $\lambda$ for the multiplication $\pi$. Define two elements $\pi^l, \pi^r \in \mathcal{P}_2$ by
    \begin{align*}
        \pi^l := \pi \circ_1 R - R \circ_1 \pi \quad \text{ and } \quad \pi^r := \pi \circ_2 R - R \circ_1 \pi.
    \end{align*}
    Then the pair $(\pi^l, \pi^r)$ is a representation of the induced multiplication $\pi_R = \pi \circ_1 R + \pi \circ_2 R + \lambda ~\! \pi $.
\end{proposition}

\begin{proof}
    We observe that
    \begin{align*}
        &\pi^l \circ_1 \pi_R - \pi^l \circ_2 \pi^l\\ 
        &= (\pi \circ_1 R - R \circ_1 \pi) \circ_1 (\pi \circ_1 R + \pi \circ_2 R + \lambda ~\! \pi)- (\pi \circ_1 R - R \circ_1 \pi) \circ_2 (\pi \circ_1 R - R \circ_1 \pi) \\
        &= (\pi \circ_1 R) \circ_1 (\pi \circ_1 R) + (\pi \circ_1 R)\circ_1 (\pi \circ_2 R) + \lambda ~\!(\pi \circ_1 R) \circ_1 \pi - (R \circ_1 \pi) \circ_1 (\pi \circ_1 R) \\ 
        & \quad- (R \circ_1 \pi) \circ_1 (\pi \circ_2 R) - \lambda ~\! (R \circ_1 \pi) \circ_1 \pi -(\pi \circ_1 R) \circ_2 (\pi \circ_1 R) + (\pi \circ_1 R) \circ_2 (R \circ_1 \pi)  \\
        &\quad + (R \circ_1 \pi) \circ_2 (\pi \circ_1 R)- (R \circ_1 \pi ) \circ_2 (R \circ_1 \pi)\\
        &= \pi \circ_1 (R \circ_1 (\pi \circ_1 R)) + \pi \circ_1 (R \circ_1 (\pi \circ_2 R)) + \pi \circ_1 (R \circ_1 (\lambda ~\! \pi))  - (R \circ_1 (\pi \circ_1 R)) \circ_2 \pi\\
        & \quad  - (R \cancel{\circ_1 (\pi \circ_1 \pi)) \circ_2 }R - (R \circ_1 (\lambda ~\! \pi))\circ_2 \pi - \pi \circ_1 ((\pi \circ_2 R)\circ_1 R)+ ((\pi \circ_2 R)\circ_1 R)\circ_2 \pi \\
        &\quad + (R\cancel{ \circ_1 (\pi \circ_2 \pi))\circ_2 }R - (R \circ_1 (\pi \circ_2 R)) \circ_2 \pi\\
        &= \pi \circ_1 \big(R \circ_1 \pi_R - (\pi \circ_2 R) \circ_1 R \big) - \big( R \circ_1 \pi_R - (\pi \circ_2 R)\circ_1 R \big)\circ_2 \pi =0.
    \end{align*}
    Similarly, we also have
    \begin{align*}
        &\pi^r \circ_1 \pi^l - \pi^l \circ_2 \pi^r \\%=(\pi \circ_2 R - R \circ_1 \pi)\circ_1 (\pi \circ_1 R - R \circ_1 \pi)- (\pi \circ_1 R - R \circ_1 \pi) \circ_2 (\pi \circ_2 R - R \circ_1 \pi)\\
        &= (\pi \circ_2 R - R \circ_1 \pi)\circ_1 (\pi \circ_1 R - R \circ_1 \pi)- (\pi \circ_1 R - R \circ_1 \pi) \circ_2 (\pi \circ_2 R - R \circ_1 \pi) \\
        &= (\pi \circ_2 R)\circ_1 (\pi \circ_1 R)- (\pi \circ_2 R)\circ_1 (R \circ_1 \pi)- (R \circ_1 \pi)\circ_1 (\pi \circ_1 R)+ (R \circ_1 \pi)\circ_1 (R \circ_1 \pi) \\
        &\quad - (\pi \circ_1 R)\circ_2 (\pi \circ_2 R) + (\pi \circ_1 R) \circ_2 (R \circ_1 \pi) + (R \circ_1 \pi) \circ_2 (\pi \circ_2 R)- (R \circ_1 \pi) \circ_2 (R \circ_1 \pi)\\
        &= ((\cancel{\pi \circ_1 \pi)\circ_1 R)\circ_3}R - ((\pi \circ_2 R)\circ_1 R)\circ_1 \pi - (R \circ_1 (\pi \circ_1 R))\circ_2 \pi + (R \circ_1 (\pi \circ_1 R))\circ_1 \pi \\
        &\quad -((\cancel{\pi \circ_2 \pi)\circ_1 R)\circ_3} R + ((\pi \circ_2 R)\circ_1 R)\circ_2 \pi + (R \circ_1 (\pi \circ_2 R))\circ_1 \pi- (R \circ_1 (\pi \circ_2 R))\circ_2 \pi\\
        &= \big(-(\pi \circ_2 R)\circ_1 R + R \circ_1 (\pi \circ_1 R + \pi \circ_2 R)  \big)\circ_1 \pi + \big((\pi \circ_2 R)\circ_1 R - R \circ_1 (\pi \circ_1 R + \pi \circ_2 R)  \big)\circ_2 \pi \\
        &= \big( -R \circ_1 (\lambda ~\! \pi)\big) \circ_1 \pi + \big( R \circ_1 (\lambda ~\! \pi) \big)\circ_2 \pi \\
        &= \lambda ~\! R \circ_1 \big( - \pi \circ_1 \pi + \pi \circ_2 \pi \big) = 0
    \end{align*}
    and
    \begin{align*}
        &\pi^r \circ_1  \pi^r - \pi^r \circ_2 \pi_R \\
        &=(\pi \circ_2 R- R \circ_1 \pi)\circ_1 (\pi \circ_2 R- R \circ_1 \pi)- (\pi \circ_2 R- R \circ_1 \pi)\circ_2 (\pi \circ_1 R + \pi \circ_2 R + \lambda ~\! \pi) \\
        &=(\pi \circ_2 R)\circ_1 (\pi \circ_2 R)- (\pi \circ_2 R)\circ_1 (R \circ_1 \pi)- (R \circ_1 \pi)\circ_1 (\pi \circ_2 R) + (R \circ_1 \pi)\circ_1 (R \circ_1 \pi) \\
        & \quad - (\pi \circ_2 R)\circ_2 (\pi \circ_1 R)- (\pi \circ_2 R) \circ_2 (\pi \circ_2 R) - \lambda ~\! (\pi \circ_2 R)\circ_2 \pi + (R \circ_1 \pi)\circ_2 (\pi \circ_1 R)\\
        &\quad + (R \circ_1 \pi) \circ_2 (\pi \circ_2 R)+ \lambda ~\! (R \circ_1 \pi)\circ_2 \pi\\
        &= \pi \circ_2 ((\pi \circ_2 R)\circ_1 R)-((\pi \circ_2 R)\circ_1 R)\circ_1 \pi - (R \cancel{\circ_1 (\pi \circ_1 \pi))\circ_2} R + (R \circ_1 (\pi \circ_1 R))\circ_1 \pi \\
        &\quad - \pi \circ_2 (R \circ_1 (\pi \circ_1 R))- \pi \circ_2 (R \circ_1 (\pi \circ_2 R))- \pi \circ_2 (R \circ_1 (\lambda ~\! \pi)) + (R \cancel{\circ_1 (\pi \circ_2 \pi))\circ_2} R\\
        &\quad  + (R \circ_1 (\pi \circ_2 R))\circ_1 \pi + (R \circ_1 (\lambda ~\! \pi))\circ_1 \pi\\
        &= \pi \circ_2 \big( (\pi \circ_2 R) \circ_1 R - R \circ_1 \pi_R \big)+ \big(-(\pi \circ_2 R) \circ_1 R + R \circ_1 \pi_R \big)\circ_1 \pi = 0.
    \end{align*}
    Hence, the result follows.
\end{proof}

As a conclusion from the above proposition, one may consider the cochain complex $\{ \mathcal{P}_\bullet, \delta_{\pi_R; \pi^l, \pi^r} \}$ associated with the multiplication $\pi_R$ with coefficients in the above-defined representation $(\pi^l, \pi^r)$. Then we have the following result.

\begin{proposition}
    For any $f \in \mathcal{P}_n$, we have $d_R (f) = -~\!\delta_{\pi_R; \pi^l, \pi^r} (f)$, where the map $d_R$ is given in (\ref{map-dr}). As a conclusion, there is an isomorphism
    \begin{align*}
    H^\bullet_{\mathcal{P}, \pi} (R) \cong H^\bullet (\pi_R; \pi^l, \pi^r ).
\end{align*}
\end{proposition}

\begin{proof}
    By a direct calculation, we have
    \begin{align*}
        &\! d_R (f) \\
        &= - ( d_\lambda f + [R,f]_\mathsf{D} )\\
        &= \lambda ~\! \iota_\pi f - [R,f]_\pi - \iota_{\theta R} f + (-1)^n ~\! \iota_{\theta f} R \\
        &= \lambda ~\! \sum_{i=1}^n (-1)^{i-1}~\! f \circ_i \pi - (\pi \circ_2 f)\circ_1 R + (-1)^n ~\!(\pi \circ_2 R) \circ_1 f + \sum_{i=1}^n (-1)^{i-1}~\! f \circ_i (\pi \circ_1 R + \pi \circ_2 R)\\
        & \quad + (-1)^n ~\!R \circ_1(-\pi \circ_1 f + (-1)^{n}\pi \circ_2 f)\\
        &= - \lambda ~\! \sum_{i=1}^n (-1)^{i}~\! f \circ_i \pi - (\pi \circ_1 R)\circ_2 f - (-1)^{n+1} ~\! (\pi \circ_2 R)\circ_1 f - \sum_{i=1}^n (-1)^{i}~\! f \circ_i (\pi \circ_1 R + \pi \circ_2 R)\\
        & \quad + (-1)^{n+1}~\! (R \circ_1 \pi)\circ_1 f + (R \circ_1 \pi) \circ_2 f\\
        &= - \big\{ (-1)^{n+1}~\! (\pi \circ_2 R - R \circ_1 \pi)\circ_1 f + (\pi \circ_1 R - R \circ_1 \pi) \circ_2 f + \sum_{i=1}^n (-1)^i~\! f \circ_i (\pi \circ_1 R + \pi \circ_2 R + \lambda ~\! \pi) \big\}\\
        &= - ~ \! \delta_{\pi_R; \pi^l, \pi^r} (f).
    \end{align*}
    Hence the result follows.
\end{proof}

%It follows from the above proposition that the cochain complex $\{ \mathcal{P}_\bullet, d_R \}$ associated with the Rota-Baxter element $R$ of weight $\lambda$ is isomorphic to the cochain complex $\{ \mathcal{P}_\bullet , \delta_{\pi_R; \pi^l, \pi^r} \}$ associated with the induced multiplication $\pi_R$ with coefficients in the representation $(\pi^l, \pi^r)$. Hence
%It follows from the above proposition that there is an isomorphism

    Next, let $R$ be a Rota-Baxter element of weight $0$ for the multiplication $\pi$. For each $n \geq 1$, we set a linear map $\Upsilon_n : \mathcal{P}_n \rightarrow \mathcal{P}_{n+1}$ by $\Upsilon_n (f) = (-1)^n ~ \! \theta f$, for $f \in \mathcal{P}_n$. Then we have
    \begin{align*}
        (\delta_{\pi_R} \circ \Upsilon_n) (f) =~& (-1)^{n+1} ~ \! [\pi_R, \theta f ]_\mathsf{GV}\\
        =~& (-1)^n ~ \! [\theta R, \theta f]_\mathsf{GV}\\
        \stackrel{(\ref{theta-D})}{=}& ~\!(-1)^n ~ \! \theta ([R, f]_\mathsf{D}) = (-1)^{n+1} ~ \! \theta (d_R f) = (\Upsilon_{n+1} \circ d_R)(f).
    \end{align*}
    Hence, we obtain the following result.

    \begin{proposition}
        Let $R$ be a Rota-Baxter element of weight $0$ for the multiplication $\pi$ in a nonsymmetric operad $\mathcal{P}$. Then there is a homomorphism $H^\bullet_{\mathcal{P}, \pi} (R) \rightarrow H^{\bullet +1}_{\mathcal{P}} (\pi_R)$ from the cohomology associated with the Rota-Baxter element $R$ to the cohomology associated with the induced multiplication $\pi_R$.
    \end{proposition}

\subsection{Averaging elements}
In this subsection, we consider averaging elements for a multiplication $\pi$ in a given nonsymmetric operad. Such elements generalize averaging operators on associative algebras \cite{das-avg,pei-guo}. Then, with the help of planar binary trees, we construct a graded Lie algebra whose Maurer-Cartan elements are precisely averaging elements for the multiplication $\pi$. %\textcolor{red}{Subsequently, we define and study the cohomology of an averaging element.}

Let $\mathcal{P}$ be a nonsymmetric operad and $\pi$ be a multiplication on $\mathcal{P}$. An element $r \in \mathcal{P}_1$ is said to be an {\bf averaging element} for the multiplication $\pi$ if 
\begin{align}
    (\pi \circ_2 r) \circ_1 r = r \circ_1 (\pi \circ_1 r) = r \circ_1 (\pi \circ_2 r).
\end{align}

\begin{remark}
    Let $\mathcal{P} = \mathrm{End}_A$ be the endomorphism operad and let $\pi \in \mathcal{P}_2 = (\mathrm{End}_A)_2$ be a multiplication on $\mathcal{P}$ representing the associative algebra structure $ (A, ~\! 
 \cdot ~ \!)$ on the vector space $A$. Then an element $r \in (\mathrm{End}_A)_1 = \mathrm{Hom}(A, A)$ is an averaging element for the multiplication $\pi$ if and only if $r$ is an averaging operator on the associative algebra $ (A, ~\! \cdot ~ \!)$ \cite{das-avg,pei-guo}, i.e.,
 \begin{align*}
     r (a) \cdot r(b) = r (r(a) \cdot b) = r (a \cdot r(b)), \text{ for all } a, b \in A.
 \end{align*}
 Thus, an averaging element for a given multiplication generalizes averaging operators on associative algebras.
\end{remark}

\begin{proposition}
    Let $r \in \mathcal{P}_1$ be an averaging element for the multiplication $\pi$ in a nonsymmetric operad $\mathcal{P}$. Then the elements $\pi_r^{\dashv} := \pi \circ_2 r$ and $\pi_r^{\vdash} := \pi \circ_1 r$ both are multiplications on $\mathcal{P}$. Moreover, they additionally satisfy the compatibilities
    \begin{align}\label{comp-di}
        \pi_r^{\dashv} \circ_2 \pi_r^\dashv = \pi_r^{\dashv} \circ_2 \pi_r^\vdash, \qquad \pi_r^{\dashv} \circ_1 \pi_r^\vdash = \pi_r^{\vdash} \circ_2 \pi_r^\dashv \quad  \text{ and } \quad   \pi_r^{\vdash} \circ_1 \pi_r^\vdash = \pi_r^{\vdash} \circ_1 \pi_r^\dashv.
    \end{align}
\end{proposition}

\begin{proof}
    We have
    \begin{align*}
        \pi_r^\dashv \circ_1 \pi_r^\dashv &= (\pi \circ_2 r)\circ_1 (\pi \circ_2 r) = ( \pi \circ_1 (\pi \circ_2 r))\circ_3 r = ((\pi \circ_1 \pi) \circ_2 r ) \circ_3 r= ((\pi \circ_2 \pi) \circ_2 r ) \circ_3 r       \\
        & =(\pi \circ_2(\pi \circ_1 r))\circ_3 r  =\pi \circ_2 ((\pi \circ_2 r)\circ_1 r)=\pi \circ_2 (r \circ_1 (\pi \circ_2 r))
        = (\pi \circ_2 r) \circ_2 (\pi \circ_2 r)\\ &  = \pi_r^\dashv \circ_2 \pi_r^\dashv
    \end{align*}
    and 
     \begin{align*}
        \pi_r^\vdash \circ_1 \pi_r^\vdash &= (\pi \circ_1 r) \circ_1 (\pi \circ_1 r)= \pi \circ_1(r \circ_1 (\pi \circ_1 r) )= \pi \circ_1 ((\pi \circ_2 r)\circ_1 r) =(\pi \circ_1 (\pi \circ_1 r))\circ_2 r  \\
        &=((\pi \circ_1 \pi)\circ_1 r)\circ_2 r= ((\pi \circ_2 \pi)\circ_1 r)\circ_2 r= ((\pi \circ_1 r)\circ_2 \pi )\circ_2 r=  (\pi \circ_1 r)\circ_2 (\pi \circ_1 r)\\& = \pi_r^\vdash \circ_2 \pi_r^\vdash.
    \end{align*}
    This shows that $\pi_r^\vdash$ and $\pi_r^\dashv$ are multiplications on $\mathcal{P}$. Similarly, one can check the compatibility conditions of (\ref{comp-di}).
\end{proof}

Let $\mathcal{P} = (\mathcal{P}, \circ, \mathbbm{1})$ be a nonsymmetric operad and $\pi$ be a multiplication on $\mathcal{P}$. In the following, we construct a graded Lie algebra whose Maurer-Cartan elements correspond to averaging elements for the multiplication $\pi$. At first, for each $ n\geq 1$, let $Y_n$ be the set of all planar binary trees with $n+1$ leaves, one root and each interior vertex trivalent. That is,

\medskip

\medskip

\begin{center}
 $ {}^{ Y_1 ~ =~ \Big\{ }$   
\begin{tikzpicture}[scale=0.2]
\draw (-4,-2)-- (-2,0); \draw (-4,-2) -- (-4,-4); \draw (-6,0) -- (-4,-2);
\end{tikzpicture}
${}^\Big\}$, \qquad 
$ {}^{ Y_2 ~ =~ \Big\{ }$ 
\begin{tikzpicture}[scale=0.2]
\draw (6,0) -- (8,-2);    \draw (8,-2) -- (10,0);     \draw (8,-2) -- (8,-4);     \draw (9,-1) -- ( 8,0);
\end{tikzpicture} ~,~ 
\begin{tikzpicture}[scale=0.2]
\draw (0,0) -- (2,-2); \draw (2,-2) -- (4,0); \draw (2,-2) -- (2,-4); \draw (1,-1) -- (2,0); 
\end{tikzpicture}
${}^\Big\}$, \qquad $ {}^{ Y_3 ~ =~ \Big\{ }$
\begin{tikzpicture}[scale=0.2]  
   \draw (36,0) -- (38, -2) ; \draw (38, -2) -- (38, -4) ; \draw (38, -2) -- (40, 0); \draw (37.33, 0) -- ( 38.67, - 1.33) ; \draw (38.66, 0) -- (39.34, -0.66);	
\end{tikzpicture} ~,~
\begin{tikzpicture}[scale=0.2]  
  \draw (30,0) -- (32,-2); \draw (32, -2) -- (32, -4); \draw (32,-2) -- (34, 0); \draw (31.33, 0) -- (32.67 , -1.33) ; \draw (32.66, 0) -- (32, -0.66); 
\end{tikzpicture} ~,~
\begin{tikzpicture}[scale=0.2] 
  \draw (24,0)-- (26,-2); \draw (26,-2)-- (26,-4); \draw (26,-2) -- (28,0); \draw (25.33, 0) -- (24.66, -0.66); \draw (26.66, 0) -- (27.34, -0.66); 
\end{tikzpicture} ~,~
\begin{tikzpicture}[scale=0.2]
 \draw (18,0) -- (20,-2); \draw (20, -2) -- (20, -4); \draw (20, -2) -- (22,0); \draw (19.33, -1.33) -- (20.66, 0); \draw (19.33, 0) -- (20, -0.66);
\end{tikzpicture} ~,~
 \begin{tikzpicture}[scale=0.2]
	     \draw (12,0)-- (14,-2); \draw (14,-2) -- (14,-4); \draw (14, -2) -- (16,0); \draw (12.7, - 0.7) -- (13.3333, 0) ; \draw (13.33, -1.33) -- (14.66, 0);
	     \end{tikzpicture} ${}^\Big\}$, \quad \text{etc.}
\end{center}

\medskip

\medskip

\noindent For each $T \in Y_n$, we label the $(n+1)$ leaves of $T$ by $\{ 0, 1, \ldots, n \}$ from left to right. Hence, for any $0 \leq i \leq n$, there is a map $d_i: Y_n \rightarrow Y_{n-1}$ obtained by deleting the $i$-th leaf. Using these maps, for any $m, n \geq 1$ and $1 \leq i \leq m$, one may define maps $R_0^{m; i, n} : Y_{m+n-1} \rightarrow Y_m$ and $R_i^{m;i, n} : Y_{m+n-1} \rightarrow Y_n$ by
\begin{align*}
    R_0^{m;i, n} (T) :=~& (d_i  d_{i+1}  ~\! \cdots ~ \! d_{i+n-2}) (T),\\
    R_i^{m; i, n} (T) :=~& (d_0  d_1 ~\! \cdots ~\! d_{i-2} d_{i + n } ~\! \cdots ~ \! d_{m+n-1}) (T),
\end{align*}
for $T \in Y_{m+n-1}$. For any $n \geq 1$, we set $\mathcal{Q}_n = \mathrm{Hom} ({\bf k}[Y_n], \mathcal{P}_n)$. Given any $F \in \mathcal{Q}_m$, $G \in \mathcal{Q}_n$ and $1 \leq i \leq m$, we define an element $F \circ^\mathcal{Q}_i G \in \mathcal{Q}_{m+n-1}$ by
\begin{align*}
    (F \circ^\mathcal{Q}_i G)(T) =  F (R_0^{m;i, n} (T)) ~ \!  \circ_i ~ \! G (R_i^{m; i, n} (T)), \text{ for } T \in Y_{m+n-1}.
\end{align*}
Then similar to \cite{das}, one can show that the collection $\mathcal{Q} = \{ \mathcal{Q}_n \}_{n \geq 1}$ endowed with the partial compositions $\circ_i^\mathcal{Q}$ defined above forms a nonsymmetric operad with the identity element $\mathbbm{1}^\mathcal{Q} \in \mathcal{Q}_1$ being given by $\mathbbm{1}^\mathcal{Q} ( ~ \! $\begin{tikzpicture}[scale=0.1]
\draw (-4,-2)-- (-2,0); \draw (-4,-2) -- (-4,-4); \draw (-6,0) -- (-4,-2);
\end{tikzpicture}$ ~ \! ) = \mathbbm{1}$. 

It is important to note that $\pi \in \mathcal{P}_2$ is a multiplication on the operad $\mathcal{P}$ if and only if the element $\pi^\mathcal{Q} \in \mathcal{Q}_2$ defined by
\begin{align*}
    \pi^\mathcal{Q} (\begin{tikzpicture}[scale=0.1]
\draw (6,0) -- (8,-2);    \draw (8,-2) -- (10,0);     \draw (8,-2) -- (8,-4);     \draw (9,-1) -- ( 8,0);
\end{tikzpicture}) =  \pi^\mathcal{Q} (\begin{tikzpicture}[scale=0.1]
\draw (0,0) -- (2,-2); \draw (2,-2) -- (4,0); \draw (2,-2) -- (2,-4); \draw (1,-1) -- (2,0); 
\end{tikzpicture}) := \pi
\end{align*}
is a multiplication on the operad $\mathcal{Q}$. We now define a map $\theta^\mathcal{Q} : \mathcal{Q}_m \rightarrow \mathcal{Q}_{m+1}$ by
\begin{align*}
    (\theta^\mathcal{Q} F)(T) = \begin{cases}
        (-1)^m ~\! (\pi^\mathcal{Q} \circ_2^\mathcal{Q} F)(T) = (-1)^m ~\! \pi \circ_2 F (R^{2;2,m}_2 (T))  & \text{ if } R_0^{2;2, m} (T) = \begin{tikzpicture}[scale=0.1]
\draw (6,0) -- (8,-2);    \draw (8,-2) -- (10,0);     \draw (8,-2) -- (8,-4);     \draw (9,-1) -- ( 8,0);
\end{tikzpicture} \\
~ - (\pi^\mathcal{Q} \circ_1^\mathcal{Q} F) (T) = 
 - \pi \circ_1 F (R^{2;1, m}_1 (T) ) & \text{ if } R_0^{2;1, m} (T) =  \begin{tikzpicture}[scale=0.1]
\draw (0,0) -- (2,-2); \draw (2,-2) -- (4,0); \draw (2,-2) -- (2,-4); \draw (1,-1) -- (2,0); 
\end{tikzpicture},
    \end{cases}
\end{align*}
for $F \in \mathcal{Q}_m$ and $T \in Y_{m+1}$. Hence, following the previous derived bracket construction and the bracket given in \cite{das-avg} controlling averaging operators on associative algebras, we define a new derived bracket $ [ \! \! [ ~, ~ ] \! \! ]_\mathsf{D}: \mathcal{Q}_m \times \mathcal{Q}_n \rightarrow \mathcal{Q}_{m+n}$ by
\begin{align*}
    [ \! \! [ F, G ] \! \! ]_\mathsf{D} =~& [F, G]_{\pi^\mathcal{Q}} + \iota_{(\theta^\mathcal{Q} F)} G - (-1)^{mn} ~ \! \iota_{(\theta^\mathcal{Q} G)} F, \\
    =~& (\pi^\mathcal{Q} \circ_2^\mathcal{Q} G) \circ_1^\mathcal{Q} F - (-1)^{mn} ~ \! (\pi^\mathcal{Q} \circ_2^\mathcal{Q} F) \circ_1^\mathcal{Q} G \\
    & \quad + \sum_{i=1}^n (-1)^{(i-1)m} ~\!  G \circ_i^\mathcal{Q} \theta^\mathcal{Q} F - (-1)^{mn} ~\! \sum_{i=1}^m (-1)^{(i-1)n} ~\!  F \circ_i^\mathcal{Q} \theta^\mathcal{Q} G,
\end{align*}
for $F \in \mathcal{Q}_m$, $G \in \mathcal{Q}_n$. Explicitly, it is given by
\begin{align}\label{avg-derived}
    [ \! \! [ F, G ] \! \! ]_\mathsf{D} (T) =~& \big( \pi \circ_2 G ( R^{2;2, n}_2 R^{n+1; 1, m}_0 (T) )   \big) \circ_1 F (R_1^{n+1;1, m}(T)) \\
    & - (-1)^{mn} ~\!  \big( \pi \circ_2 F ( R^{2;2, m}_2 R^{m+1; 1, n}_0 (T) )   \big) \circ_1 G (R_1^{m+1;1, n}(T)) \nonumber \\
    &+ \sum_{i=1}^n (-1)^{(i-1) m} ~\!  G (R^{n; i, m+1}_0 (T)) \circ_i (\theta^\mathcal{Q} F) (R_i^{n; i, m+1} (T)) \nonumber\\
    & - (-1)^{mn} ~ \! \sum_{i=1}^m (-1)^{(i-1) n} ~\!  F (R^{m; i, n+1}_0 (T)) \circ_i (\theta^\mathcal{Q} G) (R_i^{m; i, n+1} (T)), \nonumber
\end{align}
for $T \in Y_{m+n}.$ Then it turns out that $(\mathcal{Q}_\bullet , [ \! \! [ ~ , ~] \! \! ]_\mathsf{D})$ is also a graded Lie algebra. Note that an element $r \in \mathcal{P}_1$ can be regarded as an element $\widetilde{r} \in \mathcal{Q}_1$ simply by $\widetilde{r} ( \begin{tikzpicture}[scale=0.1]
\draw (-4,-2)-- (-2,0); \draw (-4,-2) -- (-4,-4); \draw (-6,0) -- (-4,-2);
\end{tikzpicture}) = r$. Then from (\ref{avg-derived}), we have
\begin{align*}
    [ \! \! [ \widetilde{r}, \widetilde{r} ] \! \! ]_\mathsf{D} ( \begin{tikzpicture}[scale=0.1]
\draw (6,0) -- (8,-2);    \draw (8,-2) -- (10,0);     \draw (8,-2) -- (8,-4);     \draw (9,-1) -- ( 8,0);
\end{tikzpicture} )  = 2 ~\! \big(  (\pi \circ_2 r) \circ_1 r - r \circ_1 (\pi \circ_2 r)   \big) ~~~ \text{ and } ~~~ 
 [ \! \! [ \widetilde{r}, \widetilde{r} ] \! \! ]_\mathsf{D} (  \begin{tikzpicture}[scale=0.1]
\draw (0,0) -- (2,-2); \draw (2,-2) -- (4,0); \draw (2,-2) -- (2,-4); \draw (1,-1) -- (2,0); 
\end{tikzpicture} )  = 2 ~\! \big(  (\pi \circ_2 r) \circ_1 r - r \circ_1 (\pi \circ_1 r)   \big).
\end{align*}
This shows that $r$ is an averaging element for the multiplication $\pi$ if and only if $\widetilde{r}$ is a Maurer-Cartan element in the graded Lie algebra $(\mathcal{Q}_\bullet, [ \! \! [ ~, ~] \! \! ]_\mathsf{D})$.

Let $\mathcal{P}$ be a nonsymmetric operad and $\pi$ be a multiplication on $\mathcal{P}$. Suppose $r$ is an averaging element for the multiplication $\pi$. For any $n \geq 1$, we define a map $d_r : \mathcal{Q}_n \rightarrow \mathcal{Q}_{n+1}$ by
\begin{align*}
    d_r : \mathcal{Q}_n \rightarrow \mathcal{Q}_{n+1} ~~~ \text{ by } ~~~ d_r (F) = - [ \! \! [ \widetilde{r}, F] \! \! ]_\mathsf{D}, \text{ for } F \in \mathcal{Q}_n.
\end{align*}
Then $\{ \mathcal{Q}_\bullet, d_r \}$ is a cochain complex whose cohomology groups are called the cohomologies of the averaging element $r$.

%By considering the nonsymmetric operad $\mathcal{Q} = (\mathcal{Q}, \circ^\mathcal{Q}, \mathbbm{1}^\mathcal{Q})$ along with the multiplication $\pi^\mathcal{Q}$, one can construct the corresponding derived bracket

\medskip

\section{Fr\"{o}licher-Nijenhuis bracket and derived bracket for some Loday-type algebras}

{In this section}, we explicitly write the Fr\"{o}licher-Nijenhuis bracket $[~,~]_\mathsf{FN}$, derived bracket $[~,~]_\mathsf{D}$ and the differential map $d_\lambda$ in the context of some Loday-type algebras.

\medskip

Let $(A, ~ \! \cdot ~ \!)$ be an associative algebra. This associative structure on $A$ is equivalent to having the multiplication $\pi \in (\mathrm{End}_A)_2$ defined by $\pi (a, b) = a \cdot b$ in the endomorphism operad. Hence, by considering the endomorphism operad $\mathcal{P} = \mathrm{End}_A$ along with this multiplication $\pi$, we can write the explicit expressions of the Fr\"{o}licher-Nijenhuis bracket and the derived brackets for the associative algebra $(A, ~\! \cdot ~\! )$. 
  %  Consider the operad $\mathrm{End}_A$ with operations $\circ_i$ defined in (\ref{ass-circ}). If the vector space $A$ has an associative multiplication $\cdot$, i.e., $(A, \cdot)$ is an associative algebra, then the element $\pi \in (\mathrm{End}_A)_2$ given by $\pi (a,b)= a\cdot b$ defines a multiplication in this operad. Using this nonsymmetric operad along with the multiplication, we get the explicit forms of the brackets $[~,~]_\mathsf{FN}$, $[~,~]_\mathsf{D} $ and the map $d_\lambda$ for an associative algebra. 
More precisely, for $f \in (\mathrm{End}_A)_m, ~ \! g \in (\mathrm{End}_A)_n$ and $a_1, \ldots , a_{m+n} \in A$, the formula (\ref{FN-bracket}) gives us
    \begin{align}\label{asso-fn-exp}
        [f,g]_\mathsf{FN}(a_1, \ldots , a_{m+n}) &= f(a_1, \ldots , a_m )\cdot g(a_{m+1}, \ldots ,a_{m+n})- (-1)^{mn} ~ \!  g(a_1, \ldots , a_n )\cdot f(a_{n+1}, \ldots , a_{m+n} )  \nonumber \\
        &\quad + (-1)^m ~ \! \sum_{i=1}^n (-1)^{(i-1)m} ~\! g ( a_1, \ldots , a_{i-1}, \delta_\pi f (a_i, \ldots , a_{i+m}), a_{i+m+1}, \ldots , a_{m+n} ) \nonumber \\
        &\quad - (-1)^{(m+1)n} ~ \! \sum_{i=1}^m (-1)^{(i-1)n} ~ \! f (a_1, \ldots , a_{i-1}, \delta_\pi g (a_i, \ldots , a_{i+n}), a_{i+n+1}, \ldots , a_{m+n} ), 
    \end{align}
    where $\delta_\pi$ is the classical Hochschild coboundary operator given by
    \begin{align*}
        \delta_\pi f (a_1, \ldots , a_{m+1}) &= a_1 \cdot f(a_2, \ldots, a_{m+1}) + (-1)^{m+1} ~\!  f (a_1, \ldots , a_m)\cdot a_{m+1}\\
        &\quad + \sum_{i=1}^m (-1)^i ~ \!  f(a_1, \ldots, a_{i-1}, a_i \cdot a_{i+1}, a_{i+2}, \ldots , a_{m+1}).
    \end{align*}
    Similarly, from (\ref{derived-bracket}) and (\ref{d_lambda}), we obtain
    \begin{align}\label{derived-das}
        &[f,g]_\mathsf{D}(a_1, \ldots , a_{m+n})  \\
        &= f(a_1, \ldots , a_m )\cdot g(a_{m+1}, \ldots ,a_{m+n})- (-1)^{mn} ~\! g(a_1, \ldots , a_n )\cdot f(a_{n+1}, \ldots , a_{m+n} ) \nonumber \\
        & - \sum_{i=1}^n (-1)^{(i-1)m} ~\! g \big (a_1, \ldots , a_{i-1}, {f(a_i, \ldots , a_{i+m-1})\cdot a_{i+m}+(-1)^{m-1} ~\!  a_i \cdot f(a_{i+1}, \ldots , a_{i+m}) },  \ldots , a_{m+n} \big) \nonumber \\
        &+(-1)^{mn} \sum_{i=1}^m (-1)^{(i-1)n} ~\!  f \big( a_1, \ldots , a_{i-1}, {g(a_i, \ldots , a_{i+n-1})\cdot a_{i+n}+(-1)^{n-1} a_i \cdot g(a_{i+1}, \ldots , a_{i+n})},  \ldots , a_{m+n} \big), \nonumber
    \end{align}
    \begin{align}
        d_\lambda f (a_1, \ldots , a_{m+1} )= \lambda \sum_{i=1}^m (-1)^i ~ \! f (a_1, \ldots , a_{i-1}, a_i \cdot a_{i+1}, a_{i+2}, \ldots , a_{m+1}).
    \end{align}
    Note that the bracket (\ref{derived-das}) was also constructed in \cite{das-rota} using Voronov's derived bracket approach.

   % \textcolor{red}{Write Both the brackets for Hom-associative}

   \medskip

   The notion of dendriform algebras was introduced by Loday \cite{loday-di} as the Koszul dual of diassociative algebras. Recall that a {\bf dendriform algebra} is a triple $(A, \prec, \succ)$ consisting of a vector space $A$ endowed with two bilinear operations $\prec, \succ : A \times A \rightarrow A$ satisfying
   \begin{align*}
       (a \prec b ) \prec c =  a \prec (b \prec c + b \succ c), \qquad (a \succ b) \prec c = a \succ (b \prec c), \qquad (a \prec b + a \succ b) \succ c = a \succ (b \succ c),
   \end{align*}
   for all $a, b, c \in A$. To understand the Fr\"{o}licher-Nijenhuis bracket and the derived brackets for a dendriform algebra, we need to first recall the nonsymmetric operad whose multiplications correspond to dendriform algebra structures on a given vector space. Let $C_n$ be the set of first $n$ natural numbers. Since the elements of $C_n$ will be treated as symbols, we write $C_n = \{[1], [2], \ldots , [n]\}$. 
% The nonsymmetric operad associated to a dendriform algebra has been studied in \cite{yau}, \cite{das-dend }. We recall the operad associated to a dendriform algebra. Let $C_n$ be the set of first n natural numbers. The elements of $C_n$ are denoted by $$ as they will be treated as certain symbols. 
    Given any vector space $A$, we set $\mathcal{P}_n= \mathrm{Hom}(\mathbf{k}[C_n]\otimes A^{\otimes n},A)$ for $n \geq 1$. For any $f \in \mathcal{P}_m$, $g \in \mathcal{P}_n$ and $1 \leq i \leq m$, one define an element $f \circ_i g \in \mathcal{P}_{m+n-1}$ by
    \medskip
        \begin{align*}
            &(f \circ_i g) ([r]; a_1, \ldots , a_{m+n-1})\\
                   & = \begin{cases}
                        f\big( \medmath{[r]}; a_1, \ldots, a_{i-1}, g( \medmath{ [1]+ \cdots +[n]}; a_i, \ldots, a_{i+n-1}), a_{i+n}, \ldots , a_{m+n-1}  \big) &\text{if~~~}{1 \leq r \leq i-1}\\
                        f\big( \medmath{[i]}; a_1, \ldots, a_{i-1}, g( \medmath{[r-i+1]}; a_i, \ldots, a_{i+n-1}), a_{i+n}, \ldots , a_{m+n-1}  \big) &\text{if~~~}  i \leq r \leq i+n-1\\
                        f\big(\medmath{ [r-n+1]}; a_1, \ldots, a_{i-1}, g( \medmath{ [1]+ \cdots +[n] }; a_i, \ldots, a_{i+n-1}), a_{i+n}, \ldots , a_{m+n-1}  \big) &\text{if~~~} {i+n \leq r \leq m+n-1},
                    \end{cases}            
        \end{align*}
\medskip
        
      \noindent  where $[r] \in C_{m+n-1}$ and $a_1, \ldots , a_{m+n-1} \in A$. Then the collection of spaces $\mathcal{P} = \{ \mathcal{P}_n \}_{n \geq 1}$ equipped with the partial compositions defined above forms a nonsymmetric operad with the identity element $\mathbbm{1} \in \mathcal{P}_1 = \mathrm{Hom} ({\bf k}[C_1] \otimes A, A)$ being given by $\mathbbm{1}([1]; a)=a$ for $a \in A$. It is important to note that an element $\pi \in \mathcal{P}_2 =  \mathrm{Hom} ({\bf k}[C_2] \otimes A^{\otimes 2}, A)$ is equivalent to have two bilinear operations $\prec, \succ : A \times A \rightarrow A$ given by $a \prec b = \pi ([1]; a, b)$ and $a \succ b = \pi ([2]; a, b)$, for $a, b \in A$. Then it turns out that $\pi$ is a multiplication on the nonsymmetric operad $\mathcal{P}$ if and only if $(A, \prec, \succ)$ is a dendriform algebra. Hence, one may explicitly write the Fr\"{o}licher-Nijenhuis bracket and the derived bracket for a given dendriform algebra $(A, \prec, \succ)$. Note that these brackets involve only the cup product, partial compositions, the differential $\delta_\pi$ and the map $\theta$. Here we shall explicitly write the cup product and the map $\theta$ only, which can be obtained using partial compositions. The differential $\delta_\pi$ can also be written explicitly using (\ref{delta-pi-lr}). More precisely, for $f \in \mathcal{P}_m$, $g \in \mathcal{P}_n$ and $[r] \in C_{m+n}$, $[s] \in C_{m+1}$,

        %If our original vector space $A$ has a dendriform algebra structure with operations $\prec, \succ$, then the element $\pi \in \mathcal{P}_2$ given by $\pi ([1];a,b)= a\prec b$, $\pi ([2];a,b)= a\succ b$ defines a multiplication in the operad $\mathcal{P}$.
        
     %   The brackets $[~,~]_\mathsf{FN}, [~,~]_\mathsf{D}$ and the map $d_\lambda$ associated to a dendriform algebra can be interpreted using only the partial compositions of the above nonsymmetric operad.

        \begin{align*}
            (f \cup_\pi g )(\medmath{ [r]; a_1, \ldots , a_{m+n}})= 
                \begin{cases}
                    f( \medmath{[r]; a_1, \ldots , a_m}) \prec g( \medmath{[1]+\cdots +[n]; a_{m+1}, \ldots ,a_{m+n}}) & \text{if~}  \medmath{1 \leq r \leq m}\\
                    f( \medmath{[1]+\cdots +[m]; a_1, \ldots ,a_m}) \succ g( \medmath{[r-m]; a_{m+1}, \ldots , a_{m+n}}) & \text{if~}  \medmath{m+1 \leq r \leq m+n},
                \end{cases}
        \end{align*}

        \medskip

            \begin{align*}
           & \theta f ([s]; a_1, \ldots , a_{m+1})\\
                &= \begin{cases}
                  ~ - f([1]; a_1, \ldots ,a_m) \prec a_{m+1} + (-1)^{m} ~ \!  a_1 \prec f([1]+ \cdots + [m]; a_2, \ldots ,a_{m+1}) & \text{if~} s=1\\
                   ~ - f([s]; a_1, \ldots ,a_m) \prec a_{m+1} + (-1)^{m} ~ \! a_1 \succ f([s-1]; a_2, \ldots ,a_{m+1}) & \text{if~} 2 \leq s \leq m\\
                    ~  - f([1]+ \cdots +[m]; a_1, \ldots ,a_m) \succ a_{m+1} + (-1)^{m} ~ \! a_1 \succ f([m]; a_2, \ldots ,a_{m+1}) & \text{if~} s=m+1.
                \end{cases}
        \end{align*}
        Finally, the map $d_\lambda$ is given by

        \begin{align*}
           %& d_\lambda f([s]; a_1, \ldots , a_{m+1})\\
              d_\lambda f([s]; a_1, \ldots , a_{m+1})
              &= \lambda~\! \bigg\{ \sum_{i=1}^{s-2} (-1)^i ~ \! f ([s-1]; a_1, \ldots, a_{i-1}, a_i \prec a_{i+1} + a_i \succ a_{i+1}, \ldots, a_{m+1}) \\
              & \qquad + \sum_{i=s-1, s} (-1)^i ~ \! f ([i]; a_1, \ldots, a_{i-1}, \pi ([s-i+1]; a_i, a_{i+1}), \ldots, a_{m+1} ) \\
              & \qquad +\sum_{i=s+1}^m (-1)^i ~ \! f ([s]; a_1, \ldots, a_{i-1}, a_i \prec a_{i+1} + a_i \succ a_{i+1}, \ldots, a_{m+1}) \bigg\}.
        \end{align*}

    \medskip

    \medskip
Next, we shall explicitly write the brackets for a given Hom-associative algebra. Let $\mathcal{P} = (\mathcal{P}, \circ, \mathbbm{1})$ be a nonsymmetric operad and $\alpha \in \mathcal{P}_1$ be an element. For each $n \geq 1$, we define
    \begin{align*}
        \mathcal{P}_n^\alpha := \{ f \in \mathcal{P}_n ~ \! | ~\! \alpha \circ_1 f =  ( \cdots (( f \circ_n \alpha) \circ_{n-1} \alpha) \cdots ) \circ_1 \alpha ~\!   \}.
    \end{align*}
    Then $\mathcal{P}^\alpha = (\mathcal{P}^\alpha, \circ^\alpha , \mathbbm{1})$ is a nonsymmetric operad, where the partial compositions $\circ_i^\alpha : \mathcal{P}_m^\alpha \times \mathcal{P}_n^\alpha \rightarrow \mathcal{P}_{m+n-1}^\alpha$ (for $m , n \geq 1$ and $1 \leq i \leq m$) are given by
    \begin{align}\label{hom-partial-comp}
        f \circ_i^\alpha g := ( \cdots (((( \cdots (f \circ_n \alpha^{n-1}) \cdots ) \circ_{i+1} \alpha^{n-1} ) \circ_i g ) \circ_{i-1} \alpha^{n-1} ) \cdots ) \circ_1 \alpha^{n-1},
    \end{align}
    for $f \in \mathcal{P}_m^\alpha$ and $g \in \mathcal{P}_n^\alpha$ \cite{das}. That is, $ f\circ_i^\alpha g$ is obtained by inserting $\alpha^{n-1}$ in all possible positions of $f$ except the $i$-th position where we insert the element $g$. The partial compositions (\ref{hom-partial-comp}) can be graphically understood by the following diagram:

   \medskip

   \medskip
    
    
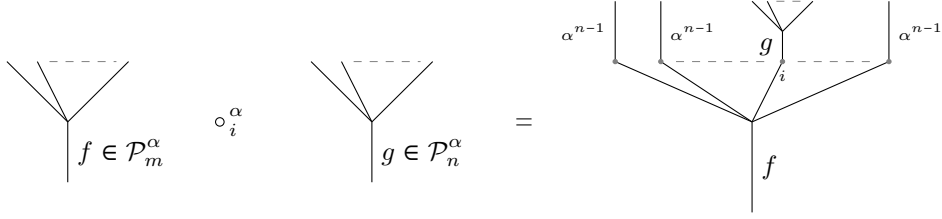
\begin{center}
        \begin{tikzpicture}[scale=.4]
            %f%
            \draw (-17,-2) -- (-15,-4); \draw (-16,-2) -- (-15,-4); \draw (-13,-2)-- (-15,-4); \draw (-15,-6)-- (-15,-4) node[midway, right]{$f \in \mathcal{P}_m^\alpha$};
            %%circle i operation%%
            \draw (-10,-4) circle (4pt) node[right]{$_i ^\alpha$};
            %g%
            \draw (-7,-2) -- (-5,-4); \draw (-6,-2) -- (-5,-4); \draw (-3,-2)-- (-5,-4); \draw (-5,-6)-- (-5,-4) node[midway, right]{$g  \in \mathcal{P}_n^\alpha$};
            %%equal%%
            \draw (-.25,-3.9) -- (.25,-3.9); \draw (-.25,-4.1) -- (.25,-4.1);
            %f%
            \draw (7.5,-4) -- (7.5,-7) node[midway, right]{$f$} ;
            \draw (7.5,-4) -- (3,-2) ; \draw (7.5,-4) -- (4.5,-2) ; \draw (7.5,-4) -- (8.5,-2) node[below]{\tiny{$i$}};  \draw  (7.5,-4)  -- (12,-2) ;
            %g%
            \draw (8.5,-2) -- (8.5,-1) node[midway, left]{$g$} ;
            \draw (8.5,-1) -- (7.5,0); \draw (8.5,-1) -- (8,0); \draw (8.5,-1) -- (9.5,0);
            %alpha%
            \draw (3,0) -- (3,-2) node[midway, left]{\tiny{$\alpha^{n-1}$}} ; \draw (4.5,0) -- (4.5,-2)node[midway, right]{\tiny{$\alpha^{n-1}$}} ;  \draw (12,0) -- (12,-2) node[midway, right]{\tiny{$\alpha^{n-1}$}};
            \filldraw [gray] (3,-2) circle (2pt); \filldraw [gray] (4.5,-2) circle (2pt); \filldraw [gray] (8.5,-2) circle (2pt); \filldraw [gray] (12,-2) circle (2pt);
            %dotted lines%
            \draw[dashed,ultra thin,gray] (5,-2) -- (8, -2);
            \draw[dashed,ultra thin,gray] (9,-2) -- (11.5, -2);
            \draw[dashed,ultra thin,gray] (8.3,0) -- (9.2, 0);
            \draw[dashed,ultra thin,gray] (-15.6,-2) -- (-13.4,-2);
            \draw[dashed,ultra thin,gray] (-5.6,-2) -- (-3.4,-2);
        \end{tikzpicture} .
    \end{center}

    \medskip

   \noindent  The nonsymmetric operad $\mathcal{P}^\alpha$ defined above is said to be obtained from $\mathcal{P}$ twisted by the element $\alpha \in \mathcal{P}_1$. It follows that a multiplication on the nonsymmetric operad $\mathcal{P}^\alpha$ is given by an element $\varpi \in \mathcal{P}_2$ satisfying 
    \begin{align*}
       \alpha \circ_1 \varpi = (\varpi \circ_2 \alpha) \circ_1 \alpha  \quad \text{ and } \quad (\varpi \circ_2 \alpha) \circ_1 \varpi = (\varpi \circ_2 \varpi) \circ_1 \alpha.
    \end{align*}
    When $\alpha = \mathbbm{1}$, the operad $\mathcal{P}^\alpha$ is same with the operad $\mathcal{P}$ itself, and a multiplication on $\mathcal{P}^\alpha$ is same with a multiplication on $\mathcal{P}$. Let $\mathcal{P} = \mathrm{End}_A$ be the endomorphism operad associated with the vector space $A$ and let $\alpha \in (\mathrm{End}_A)_1 = \mathrm{Hom}(A, A)$ be a linear map. Then $(\mathrm{End}_A^\alpha)_n = \{ f \in \mathrm{Hom} (A^{\otimes n}, A) ~ \! | ~ \! \alpha \circ f = f \circ \alpha^{\otimes n} \}$ for all $n \geq 1$. Moreover, a multiplication on the nonsymmetric operad $\mathcal{P}^\alpha = \mathrm{End}_A^\alpha$ corresponds to a bilinear operation $\cdot  : A \times A \rightarrow A$, $(a, b) \mapsto a \cdot b$ satisfying
    \begin{align}\label{hom-asso}
        \alpha (a \cdot b) = \alpha (a) \cdot \alpha (b) \quad \text{ and } \quad (a \cdot b) \cdot \alpha (c) = \alpha (a) \cdot (b \cdot c),
    \end{align}
    for all $a, b, c \in A$. The first one is called the `multiplicativity' condition, and the second one is the `Hom-associativity'. A triple $(A, ~ \! \cdot ~ \!, \alpha)$ consisting of a vector space $A$ endowed with a bilinear operation $\cdot: A \times A \rightarrow A$ and a linear map $\alpha: A \rightarrow A$ satisfying the identities in (\ref{hom-asso}) is called a {\bf Hom-associative algebra} \cite{makh-o,makh-sil}, and $\alpha$ is called the `twist map' of the Hom-associative algebra. Thus, a Hom-associative algebra structure on a vector space $A$ with the twist map $\alpha$ is equivalent to having a multiplication on the nonsymmetric operad $\mathrm{End}_A^\alpha$. 
    %In the following, we will explicitly write the Fr\"{o}licher-Nijenhuis bracket and the derived bracket for a given Hom-associative algebra with the twist map $\alpha$ (i.e. given a multiplication on $\mathrm{End}_A^\alpha$).
Hence, by considering the nonsymmetric operad $\mathrm{End}_A^\alpha$ along with the multiplication $\varpi \in (\mathrm{End}_A^\alpha)_2$ defined by $\varpi (a, b) = a \cdot b$, one obtain the explicit brackets for the given Hom-associative algebra $(A, ~ \! \cdot ~ \! , \alpha)$. More precisely, for $f \in (\mathrm{End}_A^\alpha)_m$ and $g \in (\mathrm{End}_A^\alpha)_n$, we have
  \begin{align*}
        &[f,g]_\mathsf{FN}(a_1, \ldots , a_{m+n}) \\
        &= (\alpha^{n-1} f(a_1, \ldots , a_m )) \cdot (\alpha^{m-1} g(a_{m+1}, \ldots ,a_{m+n})) \\
        & \quad  - (-1)^{mn} ~ \!   (\alpha^{m-1} g(a_1, \ldots , a_n )) \cdot (\alpha^{n-1} f(a_{n+1}, \ldots , a_{m+n} ))  \nonumber \\
        & \quad + (-1)^m ~ \! \sum_{i=1}^n (-1)^{(i-1)m} ~\! g ( \alpha^m a_1, \ldots ,  \alpha^m a_{i-1}, \delta_\varpi f (a_i, \ldots , a_{i+m}), \alpha^m a_{i+m+1} , \ldots ,  \alpha^m a_{m+n} ) \nonumber \\
        & \quad - (-1)^{(m+1)n} ~ \! \sum_{i=1}^m (-1)^{(i-1)n} ~ \! f ( \alpha^n a_1, \ldots , \alpha^n a_{i-1}, \delta_\varpi g (a_i, \ldots , a_{i+n}), \alpha^n a_{i+n+1}, \ldots , \alpha^n a_{m+n} ), 
    \end{align*}
\begin{align*}
        &[f,g]_\mathsf{D}(a_1, \ldots , a_{m+n})  \\
        &= (\alpha^{n-1} f(a_1, \ldots , a_m )) \cdot (\alpha^{m-1} g(a_{m+1}, \ldots ,a_{m+n})) \\
        & \quad  - (-1)^{mn} ~ \!   (\alpha^{m-1} g(a_1, \ldots , a_n )) \cdot (\alpha^{n-1} f(a_{n+1}, \ldots , a_{m+n} ))  \nonumber \\
        & \quad + \sum_{i=1}^n (-1)^{(i-1)m} ~\! g \big (\alpha^m a_1, \ldots , \alpha^m a_{i-1},  (\theta f) (a_i, \ldots, a_{i+m}), \alpha^m a_{i+m+1},  \ldots , \alpha^m a_{m+n} \big) \nonumber \\
        & \quad - (-1)^{mn} \sum_{i=1}^m (-1)^{(i-1)n} ~\!  f \big(\alpha^n a_1, \ldots , \alpha^n a_{i-1},  ( \theta g) (a_i, \ldots , a_{i+n}), \alpha^n a_{i+n+1},   \ldots , \alpha^n a_{m+n} \big), \nonumber
    \end{align*}
    \begin{align*}
        d_\lambda f (a_1, \ldots , a_{m+1} )= \lambda \sum_{i=1}^m (-1)^i ~ \! f \big( \alpha (a_1), \ldots , \alpha ( a_{i-1}) , a_i \cdot a_{i+1}, \alpha (a_{i+2}), \ldots , \alpha ( a_{m+1}) \big).
    \end{align*}
    Here $\delta_\varpi , ~ \! \theta : (\mathrm{End}_A^\alpha)_m \rightarrow (\mathrm{End}_A^\alpha)_{m+1}$ are the maps given by
 \begin{align*}
        \delta_\varpi f (a_1, \ldots , a_{m+1}) =~& \alpha^{m-1} (a_1) \cdot f(a_2, \ldots, a_{m+1})  + (-1)^{m+1} ~\!  f (a_1, \ldots , a_m)\cdot \alpha^{m-1} (a_{m+1})  \\
        &+ \sum_{i=1}^m (-1)^i ~ \!  f \big( \alpha (a_1), \ldots, \alpha (a_{i-1}), a_i \cdot a_{i+1}, \alpha ( a_{i+2}), \ldots , \alpha( a_{m+1}) \big),\\
        \theta f (a_1, \ldots , a_{m+1}) =~&  -  f (a_1, \ldots , a_m)\cdot \alpha^{m-1} (a_{m+1}) + (-1)^m ~ \!\alpha^{m-1} (a_1) \cdot f(a_2, \ldots, a_{m+1}).
    \end{align*}
    Note that when $\alpha = \mathrm{Id}_A$ the identity map on $A$, the above-defined Fr\"{o}licher-Nijenhuis bracket and derived bracket coincide with (\ref{asso-fn-exp}) and (\ref{derived-das}), respectively.

\medskip

Another associative-like algebraic structure that a nonsymmetric operad with multiplication can describe is the notion of associative conformal algebra \cite{bakalov,hou}. Recall that an {\bf associative conformal algebra} is a pair $(A, \cdot_\lambda \cdot)$ consisting of a $\mathbb{C}[\partial]$-module $A$ equipped with a $\mathbb{C}$-linear map (called the {\em $\lambda$-multiplication}) $\cdot_\lambda \cdot : A \otimes A \rightarrow A [ \lambda ]$, $a \otimes b \mapsto a_\lambda b$ satisfying the conformal sesquilinearity conditions $(\partial a)_\lambda b = - \lambda a_\lambda b$, ~$a_\lambda (\partial b) = (\partial + \lambda) a_\lambda b$ and the following conformal associativity:
\begin{align*}
    a_\lambda (b_\mu c) = (a_\lambda b)_{\lambda + \mu} c, \text{ for all } a, b, c \in A.
\end{align*}

Let $A$ be any arbitrary $\mathbb{C}[\partial]$-module. For any $n \geq 1$ and parameters $\lambda_1, \ldots, \lambda_{n-1}$, a $\mathbb{C}$-linear map $f: A^{\otimes n} \rightarrow A [\lambda_1, \ldots, \lambda_{n-1}], ~ a_1 \otimes \cdots \otimes a_n \mapsto f_{\lambda_1, \ldots, \lambda_{n-1}} (a_1, \ldots, a_n)$ is said to be a conformal sesquilinear map if for all $a_1, \ldots, a_n \in A$,
\begin{align*}
    f_{\lambda_1, \ldots, \lambda_{n-1}} (a_1, \ldots, \partial a_i, \ldots, a_n) =~& - \lambda_i f_{\lambda_1, \ldots, \lambda_{n-1}} (a_1, \ldots, a_n), \text{ for }  1 \leq i \leq n-1,\\
     f_{\lambda_1, \ldots, \lambda_{n-1}} (a_1, \ldots, a_{n-1}, \partial a_n) =~& (\partial + \lambda_1 + \cdots + \lambda_{n-1}) ~ \!  f_{\lambda_1, \ldots, \lambda_{n-1}} (a_1, \ldots, a_n).
\end{align*}
We set $\mathcal{P}_n = \{  f: A^{\otimes n} \rightarrow A [\lambda_1, \ldots, \lambda_{n-1}] ~ | ~ f \text{ is conformal sesquilinear}\}$.  Then for any $f \in \mathcal{P}_m$, $g \in \mathcal{P}_n$ and $1 \leq i \leq m$, one may define an element $f \circ_i g \in \mathcal{P}_{m+n-1}$ by
\begin{align*}
    &(f \circ_i g)_{\lambda_1, \ldots, \lambda_{m+n-2}} (a_1, \ldots, a_{m+n-1}) \\
    & \qquad = f_{\lambda_1, \ldots, \lambda_{i-1}, \lambda_i + \cdots + \lambda_{i+n-1}, \ldots, \lambda_{m+n-2}} \big( a_1, \ldots, a_{i-1}, g_{\lambda_i, \ldots, \lambda_{i+n-2}} (a_i, \ldots, a_{i+n-1}), \ldots, a_{m+n-1} \big),
\end{align*}
for any $a_1, \ldots, a_{m+n-1} \in A$. It has been shown \cite{hou} that the collection $\mathcal{P} = \{ \mathcal{P}_n \}_{n \geq 1}$ together with the partial compositions defined above forms a nonsymmetric operad with the identity element $\mathbbm{1} \in \mathcal{P}_1$ being given by $\mathbbm{1} (a) = a$, for $a \in A$. Moreover, an element $\pi \in \mathcal{P}_2$ (which is equivalent to having a $\mathbb{C}$-linear conformal sesquilinear map $\cdot_\lambda \cdot: A \otimes A \rightarrow A [\lambda]$ given by $a_\lambda b = \pi_\lambda (a, b)$ for $a, b \in A$) is a multiplication on the nonsymmetric operad $\mathcal{P}$ if and only if $(A, \cdot_\lambda \cdot)$ is an associative conformal algebra. Hence, by considering the operad $\mathcal{P}$ with the multiplication $\pi$, one can write the explicit forms of the Fr\"{o}licher-Nijenhuis bracket and derived bracket for the associative conformal algebra $(A, \cdot_\lambda \cdot)$.

\medskip

\medskip

    \noindent {\bf Acknowledgements.} Anusuiya Baishya would like to acknowledge the financial support from the Department of Science and Technology, Government of India, through the INSPIRE Fellowship [IF210619]. Both authors thank the Department of Mathematics, IIT Kharagpur, for providing the beautiful academic atmosphere where the research has been carried out.
    
\medskip

\noindent {\bf Conflict of interest statement.} On behalf of all authors, the corresponding author states that there is no conflict of interest.

\medskip

\noindent {\bf Data Availability Statement.} Data sharing does not apply to this article as no new data were created or analyzed in this study.

\end{document}